\documentclass[12pt, a4paper, headsepline]{scrartcl}   
\usepackage{scrpage2}    
\usepackage{scrtime}     
\typearea[5mm]{21}    
 \automark[subsection]{section}  %

\usepackage[german,czech,english]{babel}

\newcommand{\comment}[1]{}   
\newcommand{\ignore}[1]{}   

\newcommand{\gQuote}[1]{ \selectlanguage{german}\glqq \emph{#1} \grqq %
                         \selectlanguage{english}} %
\newcommand{\czQuote}[1]{\selectlanguage{czech}\glqq \emph{#1}\grqq %
                         \selectlanguage{english}}   

\usepackage{graphicx}                

\usepackage{amsmath, amssymb}
\usepackage{amsthm}

\usepackage{accents}     
\usepackage{yhmath}     

\mathchardef\ordinarycolon\mathcode`\:
\mathcode`\:=\string"8000
\begingroup \catcode`\:=\active
  \gdef:{\mathrel{\mathop\ordinarycolon}}
\endgroup

\usepackage{bbm}         

\usepackage{mathrsfs}    

\usepackage{xspace}  

\numberwithin{equation}{section}

\newenvironment{myitemize}
{
  \begin{list}{$\bullet$}
  {
    \setlength{\topsep}{\smallskipamount}
    \setlength{\itemsep}{0.0em}
    \setlength{\leftmargin}{\parindent}
    \setlength{\labelwidth}{1em}
    \setlength{\labelsep}{0.5em}
  }
}
{
  \end{list}
}
\newcommand{\itemref}[1]{\noindent \eqref{#1}}

\newcommand{\myparagraph}[1]{\noindent \textbf{#1}\quad}

\swapnumbers                    

\newtheorem{theorem}{Theorem}[section]

\newtheorem{proposition}[theorem]{Proposition}

\newtheorem{corollary}[theorem]{Corollary}

\newcounter{intro}

\theoremstyle{definition}       
\newtheorem{definition}[theorem]{Definition}

\newtheorem{example}[theorem]{Example}
\newtheorem{examples}[theorem]{Examples}

\newtheorem{remarks}[theorem]{Remarks}

\newcommand{\Sec}[1]{Section~\ref{sec:#1}}

\newcommand{\Thm}[1]{Theorem~\ref{thm:#1}}
\newcommand{\Thms}[2]{Theorems~\ref{thm:#1} and~\ref{thm:#2}}

\newcommand{\Ex}[1]{Example~\ref{ex:#1}}

\newcommand{\Exenum}[2]{Example~\ref{ex:#1}~(\ref{#2})}

\newcommand{\Prp}[1]{Proposition~\ref{prp:#1}}

\newcommand{\Def}[1]{Definition~\ref{def:#1}}

\newcommand{\Defenum}[2]{Definition~\ref{def:#1}~(\ref{#2})}

\newcommand{\abs}[2][{}]{\lvert{#2}\rvert_{{#1}}}    
\newcommand{\abssqr}[2][{}]{\lvert{#2}\rvert^2_{#1}} 
\newcommand{\bigabs}[2][{}]{\bigl\lvert{#2}\bigr\rvert_{#1}}     
\newcommand{\bigabssqr}[2][{}]{\bigl\lvert{#2}\bigr\rvert^2_{#1}}

\newcommand{\normsymb}{\|}
\newcommand{\bignormsymb}[1]{#1\|}

\newcommand{\norm}[2][{}]{\normsymb{#2}\normsymb_{{#1}}}    
\newcommand{\normsqr}[2][{}]{\normsymb{#2}\normsymb^2_{#1}} 

\newcommand{\bignormsqr}[2][{}]{\bignormsymb{\bigl}{#2}%
                                \bignormsymb{\bigr}^2_{#1}}

\newcommand{\Bignormsqr}[2][{}]{\bignormsymb{\Bigl}{#2}%
                                \bignormsymb{\Bigr}^2_{#1}}

\newcommand{\iprod}[3][{}]{\langle{#2},{#3}\rangle_{#1}}  
\newcommand{\bigiprod}[3][{}]{\bigl\langle{#2},{#3}\bigr\rangle_{#1}}

\newcommand{\set}[2]{\{ \, #1 \, | \, #2 \, \} }      
\newcommand{\bigset}[2]{\bigl\{ \, #1 \, \bigl|\bigr. \, #2 \, \bigr\} }

\DeclareMathOperator*{\dcup}   {\mathaccent\cdot\cup}

\newcommand{\map}[3]{ #1 \colon #2 \longrightarrow #3}    

\newcommand{\bd}  {\partial}          
\newcommand{\clo}[2][]{\overline{{#2}}^{#1}} 
\newcommand{\intr}[1]{\ring{{#1}}}    

\newcommand{\restr}[1]{{\restriction}_{#1}} 

\newcommand{\card}[1]{\lvert#1\rvert}   

\newcommand{\dd}    {\, \mathrm d}    
\DeclareMathOperator{\dom}    {dom}
\DeclareMathOperator{\ran}    {ran}
\DeclareMathOperator{\id}     {id}   

\DeclareMathOperator{\dvol}    {d\, vol}

\newcommand{\specsymb} {\sigma} 

\newcommand{\spec}[2][{}]   {\specsymb_{\mathrm{#1}}(#2)}
\newcommand{\bigspec}[2][{}]   {\specsymb_{\mathrm{#1}}\bigl(#2\bigr)}

\newcommand{\eps}{\varepsilon} 
\renewcommand{\phi}{\varphi}   
\renewcommand{\rho}{\varrho}   

\newcommand{\Alpha}{\mathrm A}         

\newcommand{\conj}[1]{\overline {#1}} 

\newcommand{\R}{\mathbb{R}} 
\newcommand{\C}{\mathbb{C}} 

\newcommand{\1}{\mathbbm 1}                    

\DeclareMathSymbol{\widetildesym}{\mathord}{largesymbols}{"65}
\newcommand\lowerwidetildesym[1][-1.3ex]{%
  \text{\smash{\raisebox{#1}{%
    $\widetildesym$}}}}
\newcommand\fixwidehat[1]{%
  \mathchoice
    {\accentset{\displaystyle\lowerwidetildesym[-1.3ex]}{#1}}
    {\accentset{\textstyle\lowerwidetildesym[-1.35ex]}{#1}}
    {\accentset{\scriptstyle\lowerwidetildesym[-2.0ex]}{#1}}
    {\accentset{\scriptscriptstyle\lowerwidetildesym[-2.5ex]}{#1}}
}
\newcommand{\wt}{\fixwidehat}           

\newcommand {\qf}[1]{\mathfrak{#1}}    

\newcommand{\HS}{\mathscr H}           
\newcommand{\HSaux}{\mathscr G}        
\newcommand{\anHS}{\mathscr K}        

\newcommand{\Sobsymb} {\mathsf H}      

\newcommand{\Contsymb} {\mathsf C}     
\newcommand{\Lsymb}    {\mathsf L}     
\newcommand{\lsymb}    {\ell}          

\newcommand{\Sobspace}[1][1]{\Sobsymb^{#1}} 

\newcommand{\Contspace}[1][{}]{\Contsymb^{#1}}     
\newcommand{\Lpspace}[1][p]    {\Lsymb_{#1}}     
\newcommand{\lpspace}[1][p]    {\lsymb_{#1}}     
\newcommand{\Lsqrspace}    {\Lpspace[2]}     
\newcommand{\lsqrspace}    {\lpspace[2]}          

\newcommand{\Cont}[2][{}]{\Contspace[#1]({#2})}

\newcommand{\Lsqr}[2][{}]{\Lsqrspace^{#1}({#2})} 
\newcommand{\lsqr}[2][{}]{\lsqrspace^{#1}({#2})}   

\newcommand{\Sob}[2][1]{\Sobspace [#1]({#2})}         

\newcommand{\Dir}{{\mathrm D}}              
\newcommand{\laplacian}[2][{}]{\Delta_{{#2}}^{{#1}}} 

\newcommand{\Err}{\mathrm O}

\newcommand{\dec}{\mathrm{dec}} 
\newcommand{\intl}{\mathrm{int}} 
\newcommand{\extl}{\mathrm{ext}} 

\newcommand{\quadtext}[1]{\quad\text{#1}\quad}
\newcommand{\qquadtext}[1]{\qquad\text{#1}\qquad}

\newcommand{\HDir}{H^\Dir}   

\newcommand{\HNeu}{H}   

\newcommand{\LS}{\mathscr N}           
\newcommand{\dplus}{\mathop{\dot+}}

\newcommand{\ul}[1]{\underline {#1}}   

\newcommand{\orient}[1]{\accentset{\curvearrowright}{#1}} 
\newcommand{\orul}[1]{\orient {\underline{#1}}}

\newcommand{\SDir}{S}          
\newcommand{\eSDir}{\clo    S} 

\newcommand{\WS}{\mathscr W}          

\newcommand{\Hmax}{H^{\max}}  
\newcommand{\Hmin}{H^{\min}}   

\newcommand{\wtHNeu}{\wt H}   

\newcommand{\DtN}{Di\-ri\-chlet-to-Neu\-mann\ }   

\newcommand{\aBVP} {ab\-stract boun\-da\-ry value prob\-lem\xspace}
\newcommand{\aBVPs}{ab\-stract boun\-da\-ry value prob\-lems\xspace}

\begin{document}

\title{Abstract graph-like space and vector-valued metric graphs}

\author{Olaf Post}

\date{March 30, 2016}

\maketitle


\begin{abstract}
  In this note we present some abstract ideas how one can construct
  spaces from building blocks according to a graph.  The coupling is
  expressed via boundary pairs, and can be applied to very different
  spaces such as discrete graphs, quantum graphs or graph-like
  manifolds.  We show a spectral analysis of graph-like spaces, and
  consider as a special case vector-valued quantum graphs.  Moreover,
  we provide a prototype of a convergence theorem for shrinking
  graph-like spaces with Dirichlet boundary conditions.

  \hfill \emph{Dedicated to Pavel Exner's 70th birthday.}
\end{abstract}


\noindent \emph{Keywords:}
   abstract boundary value problems, Dirichlet-to-Neumann operator,
   graph Laplacians, coupled spaces

\paragraph{Prologue.}
I got interested in graph-like spaces by a question of Vadim
Kostrykin, asking whether a Laplacian on a family of open sets
$(X_\eps)_{\eps>0}$, converging to a metric graph $X_0$ converges to
some suitable Laplacian on $X_0$.  At that time, I was not aware of
the work of Kuchment and Zeng~\cite{kuchment-zeng:01} and wrote down
some ideas.  Somehow Pavel must have heard about this; he invited me
to visit him in \v Re\v z in October 2002, just two months after the
big flood, which covered even the high-lying tracks with water,
resulting in a very reduced schedule.  At that time one had to buy the
local ticket at \emph{Praha Masarykovo n\'adra\v z\'\i} at a counter
where one was forced to pronounce the most complicated letter in Czech
language, the \czQuote{\v R} in \czQuote{\v Re\v z}. At least I got
the ticket I wanted, and enjoyed staying in this little pension
\emph{Hudec}.  \v Re\v z at night has something very special and rare
nowadays in our noise-polluted world --- \emph{Silence!}  Only the
dogs bark and from time to time, trains pass by on the other side of
the \emph{Vltava} \dots Also \v Re\v z was a good opportunity to pick
up some Czech words, as people in that little village only spoke Czech
(and sometimes a little bit German) \czQuote{M\'ate sma\v zen\'y
  s\'yr?}  --- \czQuote{Dobrou chu\v t!}  --- \czQuote{Pivo, pros\' \i
  m} \dots This invitation was the start of a very fruitful
collaboration with Pavel over many years, resulting in several
publications~\cite{exner-post:05,exner-post:07,exner-post:09,
  exner-post:13}.  Pavel inspired my research on graph-like spaces,
resulting even in an entire book~\cite{post:12}.  Pavel sometimes cites it
with the words ``\emph{\dots and then we apply the heavy German
  machinery \dots}''

\enlargethispage*{2ex}
\emph{Dear Pavel, thank you for having supported me over all the time;
  I hope you will find
 this new piece of ``heavy German machinery''
  useful for our future collaboration, and that we can continue
  working together for a long time \dots\\ 
  \quad \hspace*{\fill} V\v sechno nejlep\v s\'\i\ k
  narozenin\'am, Pavle!}

%
\section{Introduction}
\label{sec:intro}
%

The present little note shall serve as a unified approach how to work
on spaces that can be decomposed into building blocks (the
\emph{analytic} viewpoint) or that can be built up from building
blocks (the \emph{synthetic} viewpoint) according to a graph.  We will
call such spaces \emph{(abstract) graph-like spaces}. They can be
obtained in basically two different ways, depending whether the
graph-like space is decomposed into pieces indexed by \emph{vertices}
or \emph{edges}, respectively.  We call them \emph{vertex-coupled} or
\emph{edge-coupled}, respectively.  There is also a mixed case, when
one has a decomposition into parts indexed by vertices and edges (like
for thin $\eps$-neighbourhoods of embedded graphs or graph-like
manifolds in the spirit of~\cite{post:12}).  This case can be reduced
to the vertex-coupled case by considering the \emph{subdivision graph}
as underlying graph (see~\Def{subdivision} in \Sec{graphs} for
details).

In the edge-coupled case, one can also choose a suitable subspace at
each vertex determining the vertex conditions, very much in the spirit
of a quantum graph.  Indeed, one can consider edge-coupled spaces as
\emph{general} or \emph{vector-valued quantum graphs} (see
\cite{pankrashkin:06} and also~\cite{von-below-mugnolo:13} for a
another point of view).  Explaining the concept of metric and quantum
graphs in an article dedicated to Pavel would be (in his own words
\dots) \emph{to bring owls to Athens} or \emph{coal to Newcastle} or
\emph{firewood to the forest} \dots instead we refer to the book of
Berkolaiko and Kuchment~\cite{berkolaiko-kuchment:13} or
to~\cite[Sec.~2.2]{post:12}).  We define the coupling via the language
of \aBVPs.  Such a theory has been developed mostly for operators, in
order to describe (all) self-adjoint extensions of a given minimal
operator.  As we are interested only in ``geometric'' non-negative
operators such as Laplacians we find it more suitable to start with
the corresponding quadratic or energy forms.  A theory of \aBVPs
expressed entirely in terms of quadratic forms has been developed
recently under the name \emph{boundary pairs} in~\cite{post:15}, and
under the name \emph{boundary maps} in~\cite{post:12} (see
also~\cite{post:15} and references therein for related concepts, as
well as~\cite{hdss:12}, especially Ch.~3 by Arlinski\u\i).  In
particular, one has an abstract Dirichlet and Neumann operator, a
solution operator for the Dirichlet problem and a \DtN operator, see
\Sec{prelim}.

The coupling of \aBVPs in \Sec{graph-like} is --- of course --- not
new (see e.g.~Ch.~7 in~\cite{hdss:12} and references therein).  For
our graph-like spaces, the new point is the interpretation of the
coupled operators such as the Neumann or \DtN operator as a discrete
vector-valued graph Laplacian.

In \Sec{conv-graph-like} of this note, we explain the concept of a
distance of two abstract graph-like spaces based on their building
blocks (such as the vertex or edge part of a graph-like space).  This
concept can be used to show convergence of a family of \aBVPs to a
limit one.  The motivation is to give a unified approach for the
convergence of many types of (concrete) graph-like spaces such as
thick graphs, $\eps$-neighbourhoods of embedded graphs or graph-like
manifolds, including different types of boundary conditions (Neumann,
Dirichlet).

I'd like to thank the anonymous referee for very carefully reading
this ma\-nu\-script, valuable suggestions and pointing out quite a lot
of typos. I'm afraid there are still some left \dots
%
\section{Preliminaries}
\label{sec:prelim}
%

In this section we fix the notation and collect briefly some facts on
discrete graphs, as well as on \aBVPs (boundary pairs) and convergence
of operators acting in different Hilbert spaces.
\subsection{Discrete Graphs}
\label{sec:graphs}
Let $G=(V,E,\bd)$ be a countable graph, i.e., $V$ and $E$ are disjoint
and at most countable sets and $\map \bd E {V \times V}$ is a map
defining the incidence between edges and vertices, namely, $\bd
e=(\bd_-e,\bd_+e)$ is the pair of the \emph{initial} resp.\
\emph{terminal vertex} of a given edge $e \in E$.  Let
$E(V_1,V_2):=\set{e \in E}{\bd_-e \in V_1, \bd_+e \in V_2
  \;\text{or}\; \bd_+e \in V_1, \bd_-e \in V_2}$ for $V_1,V_2 \subset
V$.  We denote by $E_v =E(\{v\},V) \subset E$ the set of edges
adjacent with the vertex $v \in V$ and call the number $\deg v :=
\card {E_v}$ the \emph{degree} of a vertex $v \in V$.  We always
assume that the graph is locally finite, i.e.., that $\deg v< \infty$
for all $v \in V$ (but not necessarily uniformly bounded).  For ease
of notation, we also assume that the graph has no loops, i.e., edges
$e$ with $\bd_-e=\bd_+e$.

We use the convention that we have chosen already an
\emph{orientation} of each edge via $\bd e=(\bd_-e,\bd_+e)$, i.e., for
each edge $e$ there is not automatically an edge in $E$ with the
opposite direction.  In particular, we assume that
\begin{equation}
  \label{eq:ed-vx}
  \sum_{v \in V} \sum_{e \in E_v} a_e(v) 
  = \sum_{e \in E} \sum_{v=\bd_\pm e} a_e(v)
\end{equation}
holds for any numbers $a_e(v) \in \C$, and this also implies that
$\sum_{v \in V} \deg v = 2 \card E$ by setting $a_e(v)=1$.  We make
constant use of this reordering in the sequel.


Given a graph $G=(V,E,\bd)$, we construct another graph by introducing
a new vertex on each edge:
\begin{definition}
  \label{def:subdivision}
  Let $G=(V,E,\bd)$ a graph.  The \emph{subdivision graph}
  $SG=(A,B,\wt \bd)$ is the graph with vertex set $A=V \dcup E$
  (disjoint union) and edge set $B=\bigcup_{v \in V}\{v\} \times E_v$.
  Moreover,
  \begin{equation*}
    \map{\wt \bd} B {A \times A}, \qquad
    b=(v,e) \mapsto 
    \begin{cases}
      (\wt \bd_-b,\wt \bd_+b) = (v,e), & v=\bd_-e\\
      (\wt \bd_-b,\wt \bd_+b) = (e,v), & v=\bd_+e.
    \end{cases}
  \end{equation*}
\end{definition}

\subsection{Boundary pairs and \aBVPs}
\label{sec:bd2}

Following a good tradition (\nolinebreak\gQuote{Was interessiert mich mein
  Geschw\"atz von gestern, nichts hindert mich, weiser zu werden
  \dots}), we use a slightly different terminology than
in~\cite{post:12,post:15}; basically, we collect \emph{all} data
involved in a boundary pair and put it into a quintuple:
\begin{definition}
  \label{def:bd2}
  \indent
  \begin{enumerate}
  \item
    We say that the quintuple $\Pi:=(\Gamma,\HSaux,\qf h, \HS^1,\HS)$
    is an \emph{\aBVP} if
    \begin{myitemize}
    \item $\qf h$ is a closed, non-negative quadratic form densely
      defined in a Hilbert space $\HS$; such a form is also called
      \emph{energy form}; we endow its domain $\dom \qf h=\HS^1$ with
      norm given by $\normsqr[\HS^1] f = \qf h(f)+\normsqr[\HS] f$; we
      also say that the energy form is given by $(\qf h,\HS^1,\HS)$;
    \item $\HSaux$ is another Hilbert space and $\map \Gamma
      {\HS^1}\HSaux$ is a bounded operator, called \emph{boundary
        map}, such that $\HSaux^{1/2}:=\ran \Gamma (=\Gamma(\HS^1))$
      is dense in $\HSaux$.
    \end{myitemize}
    
  \item If, in addition, $\HS^{1,\Dir}:=\ker \Gamma$ is dense in
    $\HS$, we say that the \aBVP $\Pi$ \emph{has a dense Dirichlet
      domain}.\footnote{A pair $(\Gamma,\HSaux)$ is called
      \emph{boundary pair associated with the quadratic form $\qf h$}
      in~\cite{post:15} if $\ran \Gamma$ is dense in $\HSaux$ and
      $\ker \Gamma$ is dense in $\HS$.  If only $\ran \Gamma$ is dense
      in $\HSaux$, then $(\Gamma,\HSaux)$ is called a
      \emph{generalised boundary pair} in~\cite{post:15}.}

  \item We say that the \aBVP $\Pi$ is \emph{bounded} if $\Gamma$ is
    surjective, i.e., if $\ran \Gamma=\HSaux$.

  \item We say that the \aBVP $\Pi$ is \emph{trivial} if $\HSaux=\HS$
    and $\Gamma=\id$.
  \end{enumerate}
\end{definition}
A typical situation is $\HS=\Lsqr {X,\mu}$ and $\HSaux = \Lsqr
{Y,\nu}$, where $(X,\mu)$ and $(Y,\nu)$ are measured spaces such that
$Y \subset X$ is measurable.  The \aBVP has a dense Dirichlet domain
iff $\mu(Y)=0$.  The \aBVP is trivial iff $(X,\mu)=(Y,\nu)$ and
$\Gamma=\id$.

Given an \aBVP, we can define the following objects (details can be
found in~\cite{post:15}):
\begin{myitemize}
\item the \emph{Neumann operator} $\HNeu$ as the operator associated
  with $\qf h$;
\item the \emph{Dirichlet operator} $\HDir$ as the operator associated
  with the closed (!) form $\qf h \restr {\ker \Gamma}$ with domain
  $\HS^{1,\Dir}:=\ker \Gamma$;
\item the \emph{space of weak solutions} $\LS^1(z)=\set{h \in
    \HS^1}{\qf h(h,f)=z\iprod h f \; \forall f \in \HS^{1,\Dir}}$;
\item for $z \notin \spec \HDir$, $\HS^1=\HS^{1,\Dir} \dplus \LS^1(z)$
  (direct sum with closed subspaces); in particular, the
  \emph{Dirichlet solution operator} $\map{\SDir(z)=(\Gamma
    \restr{\LS^1(z)})^{-1}}{\ran \Gamma=\HSaux^{1/2}}{\LS^1(z)\subset
    \HS^1}$ is defined; we also set $S:= S(-1)$, i.e., the default
  value of $z$ is $-1$.
\item for $z \notin \spec \HDir$, the \emph{\DtN (sesquilinear) form}
  $\qf l_z$ is defined via $\qf l_z(\phi,\psi)=(\qf h-z\qf
  1)(S(z)\phi,S(-1)\psi)$, $\phi,\psi \in \HSaux^{1/2}$;
\item we endow $\HS^1$ with its natural norm given by $\normsqr[\HS^1]
  f = \qf h(f) + \normsqr[\HS] f$; 

\item we endow $\HSaux^{1/2}$ with the norm given by
  $\normsqr[\HSaux^{1/2}] \phi = \qf l_{-1}(\phi)=\normsqr[\HS^1]{\SDir
  \phi}$;
\item if the \aBVP is bounded, then $\HSaux^{1/2}=\HSaux$, and the two
  norms are equivalent; moreover, $\qf l_z$ is a bounded sesquilinear
  form on $\HSaux \times \HSaux$.
\end{myitemize}

For an \aBVP, one can always construct another boundary map
$\map{\Gamma'}\WS \HSaux$ which is defined on a subspace $\WS$ of
$\HS^1 \cap \dom \Hmax$, where $\Hmax:=(\Hmin)^*$ and $\Hmin:=\HDir
\cap \HNeu$ denote the maximal resp.\ minimal operator, and on which
$\Gamma'$ is bounded.  Moreover, one has the following \emph{abstract
  Green's (first) formula}
\begin{equation}
  \label{eq:green}
  \qf h(f,g)=\iprod[\HS]{\Hmax f} g + \iprod[\HSaux] {\Gamma'f}{\Gamma g}
\end{equation}
for all $f \in \WS$ and $g \in \HS^1$.

Another property is important (see~\cite{post:15} for details):
\begin{definition}
  \label{def:ell.reg}
  We say that an \aBVP $\Pi$ (or the boundary pair $(\Gamma,\HSaux)$)
  is \emph{elliptically regular} if the associated Dirichlet solution
  operator $\map{\SDir:=\SDir(-1)}{\HSaux^{1/2}}{\HS^1}$ extends to a
  bounded operator $\map \eSDir \HSaux \HS$, or equivalently, if there
  exists a constant $c>0$ such that $\norm[\HS]{\SDir \phi} \le c
  \norm[\HSaux] \phi$ for all $\phi \in \HSaux^{1/2}$.
\end{definition}

All our \aBVPs treated in this note will be elliptically regular.
They have the important property that the \DtN form $\qf l_z$ is
\emph{closed} as form in $\HSaux$ with domain $\dom \qf
l_z=\HSaux^{1/2}=\ran \Gamma$, and hence is associated with a closed
operator $\Lambda(z)$, called \emph{\DtN operator}; moreover, the
domain $\HSaux^1:=\dom \Lambda(z)$ of $\Lambda(z)$ is independent of
$z \in \C \setminus \spec \HDir$.  Another important consequence is
the following formula on the difference of resolvents: Let $z \in \C
\setminus (\spec \HNeu \cup \spec \HDir)$, then
\begin{equation}
  \label{eq:krein-res}
  (\HNeu-z)^{-1} = (\HDir-z)^{-1} + \eSDir(z)\Lambda(z)^{-1}\eSDir(\conj z)^*.
\end{equation}
As a consequence of~\eqref{eq:krein-res}, one has e.g.\ the spectral
characterisation
\begin{equation}
  \label{eq:spec-rel}
  \lambda \in \spec \HNeu \quad\iff\quad
  0 \in \spec {\Lambda(\lambda)}
\end{equation}
for all $\lambda \in \R \setminus \spec \HDir$.

\begin{examples}
  \label{ex:abvp}
  \sloppy Important examples of elliptically regular \aBVPs are the
  following:
  \begin{enumerate}
  \item
    \label{abvp.mfd}
    Let $(X,g)$ be a Riemannian manifold with compact smooth
    boundary $(Y,h)$, then
    \begin{equation*}
      \Pi=\bigl(\Gamma,\Lsqr{Y,h}, \qf h,\Sob {X,g}, \Lsqr{X,g}\bigr)
    \end{equation*}
    is an elliptically regular \aBVP with dense Dirichlet domain.
    Here, $\Gamma f = f\restr Y$ is the Sobolev trace, and the energy
    form is $\qf h(f)=\int_X \abssqr[g]{d f} \dvol_g$.

    This example is actually the godfather of the above-mentioned
    names for the derived objects: e.g.\ the Dirichlet resp.\ Neumann
    operators are actually the Dirichlet and Neumann Laplacians, the
    Dirichlet solution operator is the operator solving the Dirichlet
    problem (also called \emph{Poisson operator}), the abstract
    Green's formula~\eqref{eq:green} is the usual one with $\Gamma' f$
    being the normal outwards derivative and $\WS=\Sob[2] X$ e.g., and
    the \DtN operator has its standard interpretation.

  \item Bounded \aBVPs (i.e., \aBVPs, where $\ran \Gamma=\HSaux$, or
    equivalently, where the \DtN operator is bounded), and in
    particular \aBVPs with finite dimensional boundary space $\HSaux$,
    are elliptically regular.

  \item
    \label{abvp-graph}
    Let $G=(V,E,\bd)$ be a graph.  For simplicity, we consider only
    the normalised Laplacian here.  We define an energy form via
    \begin{equation*}
      \qf h(f)=\sum_{e \in E}\abssqr{f(\bd_+e)-f(\bd_-e)}
    \end{equation*}
    for $f \in \HS^1=\HS=\lsqr{V,\deg}$, where
    $\normsqr[\lsqr{V,\deg}] f = \sum_{v \in V} \abssqr{f(v)}\deg v$.
    Using~\eqref{eq:ed-vx} it is not hard to see that $0 \le \qf h(f)
    \le 2\normsqr[\lsqr{V,\deg}] f$.  The \emph{boundary} of $G$ is
    just an arbitrary non-empty subset $\bd V$ of $V$ (in particular,
    the degree of a ``boundary vertex'' can be arbitrary).  Set
    $\HSaux=\lsqr{\bd V,\deg}$ and $\Gamma f = f \restr {\bd V}$.
    Then $\Pi=(\Gamma,\lsqr{\bd V,\deg},\qf
    h,\lsqr{V,\deg},\lsqr{V,\deg})$ is an elliptically regular \aBVP
    without dense Dirichlet domain (see~\cite[Sec.~6.7]{post:15}).

    The Neumann operator acts as
      \begin{equation}
        \label{eq:normalised.lapl}
        (\HNeu f)(v)
        = (\laplacian G f)(v)
       := \frac1{\deg v} \sum_{e \in E_v} \bigl(f(v)-f(v_e)\bigr)
      \end{equation}
      for $v \in V$, where $v_e$ denotes the vertex adjacent with $e$
      and opposite to $v$.  The Dirichlet operator acts in the same
      way on $\lsqr{\intr V,\deg}$ where $\intr V := V
      \setminus \bd V$ are the \emph{interior} vertices (note that the
      Dirichlet Laplacian is not the Laplacian on the subgraph $\intr
      G := (\intr V,\intr E,\intr \bd)$ with $\intr E :=
      E(\intr V,\intr V)$ and $\intr \bd := \bd \restr{\intr
        E}$, as the degree is still calculated in the entire graph $G$
      and not in $\intr G$).

      Moreover, the decomposition $\HS=\lsqr {V,\deg} = \lsqr {\bd
        V,\deg} \oplus \lsqr {\intr V,\deg}=\HSaux \oplus \ker \Gamma$
      yields a block structure for $H$, namely,
      \begin{equation*}
        H = \begin{pmatrix}
             A & B\\ B^* & D
            \end{pmatrix}
      \end{equation*}
      with $\map A \HSaux \HSaux$, $\map B{\ker \Gamma}\HSaux$ and
      Dirichlet operator $\map {D=\HDir}{\ker \Gamma}{\ker \Gamma}$.
      The \DtN operator is
      \begin{equation*}
        \Lambda(z)
        = (A-z) - B (D-z)^{-1} B^*
      \end{equation*}
      provided $z \notin \spec \HDir=\spec D$.  Moreover, the second
      boundary map $\map{\Gamma'}{\WS=\lsqr{V,\deg}}{\HSaux=\lsqr{\bd
          V,\deg}}$ in Green's formula~\eqref{eq:green} is here
      \begin{equation*}
        (\Gamma'f)(v)
        = \frac 1{\deg v} \sum_{e \in E_v} \bigl(f(v)-f(v_e)\bigr),
        \qquad v \in \bd V,
      \end{equation*}
      for $f \in \WS=\lsqr{V,\deg}$, or in block structure,
      $\Gamma'=(A, B)$.

      Note that we have not excluded the extreme (or trivial) case
      $\bd V=V$ leading to a trivial \aBVP with $\Gamma=\id_{\lsqr
        V}$.  In this case, $\ker \Gamma=\{0\}$, hence $A=H$, $B=0$,
      $\HDir=D=0$ and $\spec \HDir=\emptyset$.  Moreover,
      $\Lambda(z)=H-z$.

  \item Let $X$ be a metric graph (with underlying discrete graph
    $G=(V,E,\bd)$ and edge length function $\map \ell E {(0,\infty)}$,
    $e \mapsto \ell_e$, (see~e.g.~\cite{berkolaiko-kuchment:13}
    or~\cite[Sec.~2.2]{post:12}) such that $\ell_0 = \inf_{e \in
      E}\ell_e>0$.  A bounded (hence elliptically regular) \aBVP is
    given by $\Pi=(\Gamma,\lsqr {V,\deg},\qf h,\Sob X,\Lsqr X)$, where
    $\Gamma f = f \restr V$ is the restriction of functions on $X$ to
    the set of vertices, $\qf h(f)=\int_X \abssqr{f'(x)} \dd x=\sum_{e
      \in E}\int_0^{\ell_e} \abssqr{f_e'(x_e)} \dd x_e$ and $f \in
    \Sob X=\bigoplus_{e \in E} \Sob{[0,\ell_e]} \cap \Cont X$.  In
    this case, the Neumann operator $\HNeu$ is the Laplacian with
    \emph{standard} or \emph{(generalised) Neumann} or
    \emph{Kirchhoff}\footnote{When one calls these vertex conditions
      ``\emph{Kirchhoff}'' as a coauthor of Pavel, one always ends up
      with at least a footnote (as in my first collaboration with
      Pavel~\cite{exner-post:05}).  For Pavel, the current
      conservation usually associated with this name, refers to the
      \emph{probability} current, which is preserved for any
      self-adjoint vertex condition.  Many other authors think of a
      more naive current, defined by a derivative considered as vector
      field.}  \emph{vertex conditions} and the Dirichlet operator
    $\HDir$ is the direct sum of the Dirichlet Laplacians on the
    intervals $[0,\ell_e]$, hence decoupled
    (see~\cite{post.in:08,post:12} for details).
  \end{enumerate}
\end{examples}

\subsection{Convergence of \aBVPs acting in different spaces}
\label{sec:conv-bd2}

We now define a concept of a ``distance'' $\delta$ for objects of
\aBVPs $\Pi$ and $\wt \Pi$ acting in different spaces.  One can think
of $\wt \Pi$ as being a perturbation of $\Pi$, and $\delta$ measures
quantitatively, how far away $\wt\Pi$ is from being isomorphic with
$\Pi$ (see \Ex{0-q-u-e} below for the case $\delta=0$).  The term
``convergence'' refers to the situation where we consider a family
$(\Pi_\eps)_{\eps \ge 0}$ of \aBVPs; one can think of $\wt
\Pi=\Pi_\eps$ and $\Pi=\Pi_0$ with ``distance'' $\delta_\eps$.  If
$\delta_\eps \to 0$ as $\eps \to 0$ then we say that \emph{$\Pi_\eps$
  converges to $\Pi_0$}.  Details of this concept of a ``distance''
between operators acting in different spaces can also be found
in~\cite[Ch.~4]{post:12}.

To be more precise, let $\Pi=(\Gamma,\HSaux,\qf h,\HS^1,\HS)$ and $\wt
\Pi=(\wt \Gamma,\wt \HSaux,\wt {\qf h},\wt \HS^1,\wt \HS)$ be two
\aBVPs.  Recall that $\HS^1$ is the domain of a closed non-negative
form $\qf h$ in the Hilbert space $\HS$, and that $\map \Gamma
{\HS^1}\HSaux$ is bounded with dense range, and similarly for the
tilded objects.  We need bounded operators
\begin{subequations}
  \label{eq:id-ops}
  \begin{equation}
    \label{eq:id-ops.hs}
    \map J \HS {\wt \HS}, \quad
    \map {J'}{\wt \HS} \HS, \quad
    \map I \HSaux {\wt \HSaux}  \quadtext{and}
    \map {I'} {\wt \HSaux} \HSaux,
  \end{equation}
  called \emph{identification operators} which replace unitary or
  isomorphic operators.  The quantity $\delta>0$ used later on
  measures how far these operators differ from isomorphisms.  We also
  need \emph{identification operators} on the level of the energy form
  domains, namely
  \begin{equation}
    \label{eq:id-ops-forms}
    \map {J^1}{\HS^1} {\wt \HS^1}   \quadtext{and}
    \map {J^{\prime 1}}{\wt \HS^1} {\HS^1}.
  \end{equation}
\end{subequations}
In contrast to~\cite{behrndt-post:pre16} we will not assume in this
note that the identification operators $I$ and $I'$ on the boundary
spaces $\HSaux$ and $\wt \HSaux$ also respect the form domains
$\HSaux^{1/2}$ and $\wt \HSaux^{1/2}$ of the \DtN operators.

We start with the energy forms and boundary maps:
\begin{definition}
  \label{def:forms.close}
  Let $\delta>0$.  We say that the energy forms $\qf h$ and
  $\wt{\qf h}$ are \emph{$\delta$-close} if there are identification
  operators $J^1$ and $J^{\prime 1} $ as in~\eqref{eq:id-ops} such
  that
    \begin{gather*}
      \bigabs{\wt{\qf h}(J^1 f, u) - \qf h(f, J^{\prime 1}u)}
      \le \delta \norm[\wt \HS^ 1] u \norm[\HS^1] f
    \end{gather*}
    holds for all $f \in \HS^1$ and $u \in \wt\HS^1$.
\end{definition}

\begin{definition}
  \label{def:bd-maps.close}
  Let $\delta>0$.  We say that the boundary maps $\Gamma$ and $\wt
  \Gamma$ are \emph{$\delta$-close} if there exist identification
  operators $J^1$, $J^{\prime 1}$, $I$ and $I'$ as
  in~\eqref{eq:id-ops} such that
  \begin{equation*}
    \norm[\wt \HSaux]{(I \Gamma - \wt \Gamma J^1)f} 
    \le \delta\norm[\HS^1] f
    \qquadtext{and}
    \norm[\HSaux]{(I'\wt \Gamma - \Gamma J^{\prime 1})u} 
    \le \delta\norm[\wt \HS^1] u
  \end{equation*}
  hold for all $f \in \HS^1$ and $u \in \wt\HS^1$.
\end{definition}

So far, we have only dealt with forms and their domains.  Let us now
define the following compatibility between the identification
operators on the Hilbert space and the energy form level:
\begin{definition}
  \label{def:id-ops-q-u-e}
  We say that the identification operators $J$, $J'$, $J^1$ and
  $J^{\prime 1}$ 
  are
  \emph{$\delta$-quasi-unitarily equivalent} with respect to the
  energy forms $\qf h$ and $\wt{\qf h}$ if
  \begin{gather*}
    \abs{\iprod[\wt \HS]{J f} u - \iprod[\HS] f
      {J'u}} \le \delta \norm[\HS] f \norm[\wt\HS] u,\\
    \norm[\HS]{f-J'Jf} \le \delta \norm[\HS^1] f, \qquad
    \norm[\wt\HS]{u-JJ'u} \le \delta \norm[\wt \HS^1] u,  \\
    \norm[\wt \HS]{J^1 f - Jf} \le \delta \norm[\HS^1] f 
    \quadtext{and}
    \norm[\HS]{J^{\prime 1} u - J'u} \le \delta \norm[\wt \HS^1] u 
  \end{gather*}
  hold for $f$ and $u$ in the respective spaces.  We say that the
  forms $\qf h$ and $\wt{\qf h}$ are \emph{$\delta$-quasi-unitarily
    equivalent}, if they are $\delta$-close with
  $\delta$-quasi-unitarily equivalent identification operators.
\end{definition}
For the boundary identification operators $I$ and $I'$ we define:
\begin{definition}
  \label{def:bd-id-ops-q-u-e}
  We say that the identification operators $I$ and $I'$ are
  \emph{$\delta$-quasi-isomorphic} with respect to the \aBVPs $\Pi$
  and $\wt \Pi$ if
  \begin{equation*}
    \norm[\HS]{\phi-I'I\phi} \le \delta \norm[\HSaux^{1/2}] \phi, \qquad
    \norm[\wt\HS]{\psi-II'\psi} \le \delta \norm[\wt \HSaux^{1/2}] \psi
  \end{equation*}
  hold for $\phi\in \HSaux^{1/2}$ and $\psi \in \wt\HSaux^{1/2}$.  We
  say that the boundary maps $\Gamma$ and $\wt \Gamma$ are
  $\delta$-quasi-isomorphic if they are $\delta$-close with
  $\delta$-quasi-unitarily equivalent $J$, $J'$, $J^1$ and $J^{\prime
    1}$ resp.\ $\delta$-quasi-isomorphic $I$ and $I'$.
\end{definition}
The $\delta$-quasi-isomorphy only refers to the \DtN form $\qf l_{-1}$
in $z=-1$ as $\normsqr[\HSaux^{1/2}]\phi = \qf
l(\phi)=\normsqr[\HS^1]{\SDir(-1)\phi}$ and no other structure of
$\Pi$; a similar note holds for $\wt \Pi$.  We do not assume that
$I^*$ is closed to $I'$, as this is too restrictive for
\Def{abvps.close} (see e.g.\ the proof of \Prp{triv.conv}: $I^*=I'$
would mean $\gamma=1$).

Finally, we define what it means for \aBVPs to be ``close'' to each
other, by combining the last four definitions:
\begin{definition}
  \label{def:abvps.close}
  Let $\delta>0$.  We say that the \aBVPs $\Pi$ and $\wt \Pi$ are
  \emph{$\delta$-quasi-isomorphic} if there exist
  $\delta$-quasi-unitarily equivalent identification operators $J$,
  $J'$, $J^1$ and $J^{\prime 1}$ and $\delta$-quasi-isomorphic
  identification operators $I$ and $I'$ for which $\qf h$ and $\wt{\qf
    h}$, respectively, $\Gamma$ and $\wt \Gamma$ are $\delta$-close.
\end{definition}

Let us illustrate this concept in two examples.

\begin{example}
  \label{ex:0-q-u-e}
  A good test for a reasonable definition of a ``distance'' is the
  case $\delta=0$: if $\Pi$ and $\wt \Pi$ are $0$-quasi-isomorphic
  then $J$ is unitary with adjoint $J'$; $J^1$ and $J^{\prime 1}$ are
  restrictions of $J$ and $J^*$, respectively.  Moreover, $J$
  intertwines $H$ and $\wt H$ in the sense that
  $J(\HNeu+1)^{-1}=(\wtHNeu+1)^{-1}J$; and $I$ is a bi-continuous
  isomorphism with inverse $I'$, and $\Gamma$ and $\wt \Gamma$ are
  equivalent in the sense that $\wt \Gamma=I\Gamma J^{\prime 1}$.  We
  call such \aBVPs \emph{isomorphic}.
\end{example}

Another rather trivial case is the following: it nevertheless plays an
important role in the study of shrinking domains like an
$\eps$-homothetic vertex neighbourhood shrinking to a point in the
limit $\eps \to 0$ (i.e., we use the \aBVP $\wt \Pi=\Pi_\eps$
associated with a compact and connected manifold $X$ of dimension $d
\ge 2$ with boundary $Y=\bd X$ and metric $\eps^2 g$ as in
\Exenum{abvp}{abvp.mfd}; in this case, $\delta=\Err(\sqrt \eps)$,
see~\cite[Sec.~5.1.4]{post:12} for details, also for the validity
of~\eqref{eq:bd.map.estimate}):
\begin{proposition}
  \label{prp:triv.conv}
  Assume that $\wt \Pi=(\wt \Gamma,\wt\HSaux,\wt{\qf h},\wt \HS^1,\wt
  \HS)$ is an \aBVP such that the corresponding Neumann operator
  $\wtHNeu$ has $0$ as simple and isolated eigenvalue in its spectrum.
  Assume also that there is $a
  \in (0,1]$ such that
  \begin{equation}
    \label{eq:bd.map.estimate}
    \normsqr[\wt\HSaux]{\wt \Gamma u}
    \le a \wt{\qf h}(u) 
      + \frac 2 a \normsqr[\wt \HS] u
  \end{equation}
  holds for all $u \in \wt \HS^1$ .

  Moreover, let $\Pi=(\id,\C,0,\C,\C)$ be a trivial \aBVP.  Then $\wt
  \Pi$ and $\Pi$ are $\delta$-quasi-isomorphic with $\delta$ depending
  only on parameters of $\wt \Pi$ and $a$,
  see~\eqref{eq:triv.conv.delta} for a precise definition.
\end{proposition}
\begin{proof}
  Let $\Phi_0$ be a normalised eigenvector associated with the
  eigenvalue $0$ of $\wt H$.  As $0 \in \spec {\wt H}$, we also have
  $0 \in \spec{\wt \Lambda(0)}$ with eigenvector $\Psi_0 = \wt \Gamma
  \Phi_0$ (see \cite[Thm.~4.7~(i)]{post:15}).  In particular, $\gamma
  := \norm[\wt \HSaux]{\wt \Gamma \Phi_0}^{-2}$ is defined.  For the
  identification operators, we set
  \begin{equation*}
    Jf = f\Phi_0, \quad 
    J^1f=Jf, \quad 
    J'u=J^*u=\iprod[\wt \HS] u {\Phi_0},\quad 
    J^{\prime 1} u=J^*u, \quad
    I\phi= \phi \Psi_0
  \end{equation*}
  and $I'=\gamma I^*$, where $I^* \psi=\iprod[\wt\HSaux] \psi
  {\Psi_0}$.  The choice of $\gamma$ implies that $I'I\phi=\phi$, and
  \begin{equation*}
    \normsqr[\wt \HSaux]{\psi - II'\psi}
    =\bignormsqr[\wt \HSaux]
       {\psi - \gamma \iprod[\wt \HSaux] \psi{\Psi_0}\Psi_0}
    \le \frac 1 {\mu_1} \wt{\qf l}_0(\psi)
  \end{equation*}
  as $\sqrt \gamma \Psi_0$ is a normalised eigenfunction of $\wt
  \Lambda(0)$ corresponding to the eigenvalue $0$, where $\mu_1:=
  d(\spec{\wt \Lambda(0)} \setminus\{0\},0)$ and $\wt{\qf l}_0$ is the
  associated quadratic form.  As $\lambda \mapsto \wt{\qf l}_\lambda$
  is monotonously decreasing (see~\cite[Thm.~2.12(v)]{post:15}), we
  have the estimate $\wt{\qf l}_0(\psi)\le \wt{\qf l}_{-1}(\psi)=:
  \normsqr[\wt \HSaux^{1/2}] \psi$.  In particular, $I$ and $I'$ are
  $(1/\sqrt{\mu_1})$-quasi-isomorphic, see \Def{bd-id-ops-q-u-e}.

  For the $\delta$-closeness of the forms resp.\ the
  boundary maps we have
  \begin{gather*}
    \wt{\qf h}(J^1 f, u) - \qf h(f, J^{\prime 1}u)=0,\\
    (I \Gamma - \wt \Gamma J^1)f
    =I f-\wt \Gamma f\Phi_0
    =f\cdot (\Psi_0 - \wt \Gamma \Phi_0)=0
    \quadtext{and}\\
    \begin{split}
      (I'\wt \Gamma - \Gamma J^{\prime 1})u 
      &=\iprod[\wt\HSaux]{\gamma \wt \Gamma u}{\Psi_0}
           - \iprod[\wt\HS] u {\Phi_0}\\
      &=\iprod[\wt\HSaux]{\gamma \wt \Gamma u}{\Psi_0}
           - \gamma \iprod[\wt\HSaux]{\Psi_0}{\Psi_0}
                \iprod[\wt\HS] u {\Phi_0}\\
      &=\gamma \bigiprod[\wt\HSaux]{
        \wt \Gamma 
          \bigl(u -\iprod[\wt \HS] u{\Phi_0}\Phi_0
        \bigr)}{\Psi_0}.
    \end{split}
  \end{gather*}
  The latter inner product can be estimated in squared absolute value
  by
  \begin{align*}
    \abssqr{(I'\wt \Gamma - \Gamma J^{\prime 1})u}
    &\le \gamma \bignormsqr[\wt \HSaux]
      {\wt \Gamma(u -\iprod[\wt \HS] u{\Phi_0}\Phi_0)}\\
    &\le \gamma 
      \Bigl(a \wt{\qf h}(u) 
         + \frac 2a \bignormsqr[\wt \HS] {u -\iprod[\wt \HS] u{\Phi_0}\Phi_0}
      \Bigr)\\
    &\le \gamma \Bigl(a + \frac 2{a \lambda_1}\Bigr) \wt {\qf h}(u),
  \end{align*}
  using~\eqref{eq:bd.map.estimate}, where $\lambda_1:= d(\spec{\wt H}
  \setminus\{0\},0)$.  Note that $\wt {\qf h}\bigl(u -\iprod[\wt \HS]
  u{\Phi_0}\Phi_0\bigr)=\wt {\qf h}(u)$ as $\wt H\Phi_0=0$ and hence
  $\wt {\qf h}(w,\Phi_0)=\iprod w {\wt H\Phi_0}=0$ for any $w \in \wt
  \HS^1$.

  Finally, $J^*Jf=f$ and $\normsqr[\wt \HS]{u-JJ^*u}=\normsqr[\wt
    \HS]{u-\iprod[\wt \HS] u {\Phi_0} {\Phi_0}} \le \frac 1{\lambda_1}
  \wt{\qf h}(u)$.  Therefore we can choose
  \begin{equation}
    \label{eq:triv.conv.delta}
      \delta
      = \max \Big\{
        \frac 1{\sqrt{\mu_1}},
        \frac 1{\norm[\wt \HSaux]{\wt \Gamma \Phi_0}}
          \sqrt{a + \frac 2{a \lambda_1}},
        \frac 1{\sqrt{\lambda_1}}
      \Bigr\}.
        \qedhere
  \end{equation}
\end{proof}

%
\section{Abstract graph-like spaces}
\label{sec:graph-like}
%

Let us first explain the philosophy briefly.  In the below-mentioned
different couplings of \aBVPs according to a graph, we show that the
Neumann operator is coupled, while the Dirichlet operator is always a
direct sum of the building blocks, i.e., decoupled.  Moreover, we give
formulas how the coupled operators can be calculated from the building
blocks.  We also analyse how the coupled operators such as the \DtN
operator resemble discrete Laplacians on the underlying or related
graphs, allowing a deeper understanding of the problem and relating it
to problems of graph Laplacians.

In particular, the resolvent formula~\eqref{eq:krein-res} gives an
expression of a globally defined object, namely the coupled Neumann
operator in terms of objects from the building blocks (see e.g.~the
formulas for $\HDir$, $\SDir(z)$ and $\Lambda(z)$ in
\Thms{vx-coupling}{ed-coupling}).  Hence the understanding of the
nature how $\Lambda(z)$ is obtained from the building blocks is
essential in understanding the global operator $\HNeu$.

\subsection{Direct sum of \aBVPs}
\label{sec:dir-sum}
Given a family $(\Pi_\alpha)_{\alpha \in \Alpha}$ of \aBVPs, we define
the \emph{direct sum} via\footnote{The direct sum of Hilbert spaces
  always refers to the Hilbert space closure of the algebraic direct
  sum in this note.}
\begin{equation*}
  \bigoplus_{\alpha \in \Alpha} \Pi_\alpha
  :=\Bigl(
    \bigoplus_{\alpha \in \Alpha} \Gamma_\alpha,
    \bigoplus_{\alpha \in \Alpha} \HSaux_\alpha,
    \bigoplus_{\alpha \in \Alpha} \qf h_\alpha,
    \bigoplus_{\alpha \in \Alpha} \HS^1_\alpha,
    \bigoplus_{\alpha \in \Alpha} \HS_\alpha
   \Bigr).
\end{equation*}
The direct sum is an \aBVP provided $\sup_{\alpha \in \Alpha}
\norm{\Gamma_\alpha} < \infty$.  As the direct sum is not coupled, we
also call them \emph{decoupled} and write
\begin{equation*}
  \Pi^\dec
  = \bigl(\Gamma^\dec, \HSaux^\dec, \qf h^\dec, \HS^{1,\dec}, \HS^\dec\bigr)
  := \bigoplus_{\alpha \in \Alpha} \Pi_\alpha.
\end{equation*}
All derived objects such as the Dirichlet solution operator or the
\DtN operator are also direct sums of the correspondent objects.

\subsection{Vertex coupling}
\label{sec:vx-coupling}
We now construct a new space from building blocks associated with each
vertex.  Let $G=(V,E,\bd)$ be a graph.  For each \emph{vertex} $v \in
V$ we assume that there is an \aBVP $\Pi_v=(\Gamma_v,\HSaux_v,\qf h_v,
\HS^1_v,\HS_v)$.

\begin{definition}
  \label{def:vx-coupling}
  We say that the family of \aBVPs $(\Pi_v)_{v \in V}$ allows a
  \emph{vertex coupling}, if the following holds:
  \begin{enumerate}
  \item
    \label{def.vx-coupling.a}
    Assume that $\sup_{v \in V} \norm {\Gamma_v} < \infty$.
  \item
    \label{def.vx-coupling.b}
    Assume there is a Hilbert space $\HSaux_e$ and a bounded operator
    $\map{\pi_{v,e}}{\HSaux_v}{\HSaux_e}$ for each edge $e \in E$.
    Let $\HSaux_v^{\max}:= \bigoplus_{e \in E_v} \HSaux_e$ and
    $\map{\iota_v}{\HSaux_v}{\HSaux_v^{\max}}$, $\iota_v
    \phi_v=(\pi_{v,e} \phi_v)_{e \in E_v}$.  We assume that $\iota_v$
    is an isometric embedding.
 
  \item
    \label{def.vx-coupling.d}
    Assume that $\ran \Gamma_{\bd_-e,e}=\ran
    \Gamma_{\bd_+e,e}=: \HSaux_e^{1/2}$ for all edges $e \in
    E$, where $\map{\Gamma_{v,e}:=\pi_{v,e}
      \Gamma_v}{\HS^1_v}{\HSaux_e}$.
  \end{enumerate}
  We set $\map{\pi_v:=\iota_v^*}{\HSaux_v^{\max}}{\HSaux_v}$,
  then $\pi_v \psi = \sum_{e \in E_v} \pi_{v,e}^* \psi_e$.  We say
  that the vertex coupling is \emph{maximal} if $\iota_v$ is
  surjective (hence unitary).  In this case, we often identify
  $\HSaux_v$ with $\HSaux_v^{\max}$.
\end{definition}

We start with an example with maximal vertex coupling spaces $\HSaux_v
\cong \HSaux_v^{\max}$ (an example with non-maximal vertex coupling
spaces will be given in \Sec{vx-ed-coupling}):
\begin{example}
  \label{ex:vx-coupling}
  Assume that we have a graph-like manifold $X$ (without edge
  contributions), i.e., $X=\bigcup_{v \in V} X_v$ such that $X_v$ is
  closed in $X$ and $Y_e := X_{\bd_-e} \cap X_{\bd_+e}$ is a
  smooth submanifold.  Then $\HS_v=\Lsqr {X_v}$, $\qf
  h_v(f)=\int_{X_v} \abssqr {df}$, $\dom \qf h_v = \HS_v^1=\Sob{X_v}$
  and $\HSaux_v=\Lsqr{\bd X_v}$, $\Gamma_v f=f \restr{\bd X_v}$; and
  $\Pi_v=(\Gamma_v,\HSaux_v,\qf h_v,\HS_v^1,\HS_v)$ is an \aBVP.
  Moreover, $(\Pi_v)_v$ allows a vertex coupling with
  $\HSaux_e=\Lsqr{Y_e}$ with maps $\map{\pi_{v,e}}{\Lsqr{\bd
      X_v}}{\Lsqr{Y_e}}$ being the restriction of a function on $\bd
  X_v$ onto one of its components $Y_e \subset \bd X_v$.  Note that
  $\ran \Gamma_{\bd_\pm e,e}=\Sob[1/2] {Y_e}$.  As
  \begin{equation*}
    \HSaux_v^{\max}
    := \bigoplus_{e \in E_v} \HSaux_e
    =\bigoplus_{e \in E_v} \Lsqr{Y_e}
    =\Lsqr{\bd X_v}
    =\HSaux_v,
  \end{equation*}
  the vertex coupling is \emph{maximal}.
  Condition~\itemref{def.vx-coupling.a} is typically fulfilled, if the
  length of each end of a building block $X_v$ is bounded from below
  by some constant $\ell_0/2>0$.
\end{example}

We construct an \aBVP $\Pi=(\Gamma,\HSaux,\qf h, \HS^1,\HS)$ from
$(\Pi_v)_v$ as follows:
\begin{align*}
  \HS := &\bigoplus_{v \in V} \HS_v, \qquad 
  \HS^{1,\dec} := \bigoplus_{v \in V} \HS^1_v, \qquad
  \qf h^\dec := \bigoplus \qf h_v;\\ 
  \HS^1 :=&
  \bigset{f=(f_v)_{v \in V} \in \HS^{1,\dec}}
         {\forall e \in E, \; \bd e=(v,w) \colon
             \; \Gamma_{v,e} f_v = \Gamma_{w,e} f_w =: \Gamma_e f},\\ 
  \qf h := &\qf h^\dec \restr {\HS^1}, \qquad 
  \HSaux := \bigoplus_{e \in E} \HSaux_e, \qquad
  \map \Gamma {\HS^1} \HSaux, \quad
  \Gamma f=(\Gamma_{\bd_\pm e, e} f_{\bd_\pm e})_{e \in E}.
\end{align*}
Denote by $\iota$ the map
\begin{equation*}
 \map \iota {\HSaux=\bigoplus_{e \in E}\HSaux_e}
      {\HSaux^\dec=\bigoplus_{v \in V}\HSaux_v}, \qquad
      (\iota \phi)_v = \pi_v \ul \phi(v) =
      \sum_{e \in E_v} \pi_{v,e}^* \phi_e,
\end{equation*}
where $\ul \phi(v)=(\phi_e)_{e \in E_v} \in \HSaux_v^{\max}$ (see
\Def{vx-coupling} for the notation).  It is easy to see that
$\map{\iota^*} {\HSaux^\dec}\HSaux$ acts as
\begin{equation*}
  (\iota^*\psi)_e = \sum_{v = \bd_\pm e} \pi_{v,e}\psi(v)
\end{equation*}
\begin{theorem}
  \label{thm:vx-coupling}
  Assume that $(\Pi_v)_{v \in V}$ is a family of \aBVPs allowing a
  \emph{vertex coupling}, then the following holds:
  \begin{enumerate}
  \item
    \label{vx-coupling.a}
    The quintuple $\Pi=(\Gamma,\HSaux,\qf h, \HS^1,\HS)$ as
    constructed above is an \aBVP.
  \item
    \label{vx-coupling.b}
    We have $\ker \Gamma=\bigoplus_{v \in V} \ker \Gamma_v$, $\HDir =
    \bigoplus_{v \in V} \HDir_v$ (i.e., the Dirichlet operator is
    decoupled) and $\spec \HDir = \clo{\bigcup_{v \in V} \spec
      {\HDir_v}}$.
  \item
    \label{vx-coupling.c}
    In particular, if all \aBVPs $\Pi_v$ have dense Dirichlet domain
    then $\Pi$ also has dense Dirichlet domain.

  \item
    \label{vx-coupling.c'}
    The Neumann operator is coupled and $f \in \dom \HNeu$ iff $f \in
    \bigoplus_{v \in V} \WS_v$ and
    \begin{equation*} 
      \Gamma_{v,e}f_v=\Gamma_{w,e}f_w, \quad
      \Gamma'_{v,e} f_v + \Gamma'_{w,e} f_w=0
      \qquad \forall e \in E, \;\bd e = (v,w)
    \end{equation*}
    with $\HNeu f= (\Hmax f)_{v \in V}$ (see~\eqref{eq:green} for the
    notation).
  \item
    \label{vx-coupling.d}
    We have $S(z)\phi = S^\dec(z) \iota \phi = \bigoplus_{v \in V}
    S_v(z)\ul \phi(v)$ for $z \notin \spec \HDir$.

  \item
    \label{vx-coupling.e}
    Moreover, if all $\Pi_v$ are elliptically regular such that
    $\sup_{v \in V}\norm[\HSaux_v\to\HS_v]{S_v} < \infty$, then $\Pi$
    is also elliptically regular.
  \item
    \label{vx-coupling.f}
    We have $\Lambda(z) = \iota^* \Lambda^\dec(z) \iota$ for $z \in \C
    \setminus \spec \HDir$, i.e.,
    \begin{equation*}
      \bigl(\Lambda(z) \phi)\bigr)_e
      = \sum_{v =\bd_\pm e} \pi_{v,e} 
          \bigl(\Lambda_v(z) \ul \phi(v)\bigr).
    \end{equation*}
  \end{enumerate}
\end{theorem}
\begin{proof}
  \itemref{vx-coupling.a}~The space $\HS^1$ is closed in
  $\HS^{1,\dec}$ as intersection of the closed spaces
  \begin{equation*}
    \set{f \in \HS^{1,\dec}} {\Gamma_{\bd_-e,e} f_{\bd_-e} =
      \Gamma_{\bd_+e,e} f_{\bd_+e}}
  \end{equation*}
  (note that since $\Gamma_{v,e}$ are bounded operators, the latter
  sets are closed).  Hence $\HS^1$ is closed and $\qf h$ is a closed
  form.  Moreover, the operator $\Gamma$ is bounded, as
  \begin{equation*}
    \normsqr[\HSaux] {\Gamma f}
    = \sum_{e \in E} \normsqr[\HSaux_e]{\Gamma_e f}
    = \frac 12\sum_{v \in V} \sum_{e \in E_v}
            \normsqr[\HSaux_e]{\Gamma_{v,e} f_v}
    = \frac 12 \sum_{v \in V} \normsqr[\HSaux_v^{\max}]{\iota_v \Gamma_v f_v}
  \end{equation*}
  using~\eqref{eq:ed-vx} and this can be estimated by $\sup_v
  \normsqr{\Gamma_v} \normsqr[\HS^1] f/2$ as $\iota_v$ is isometric.

  Finally, as $\ran \Gamma_v=\HSaux^{1/2}_v$ is dense in $\HSaux_v$ we
  also have that $\ran \Gamma_{v,e}$ is dense in $\HSaux_{v,e}$
  (applying the bounded operator $\map{\pi_{v,e}}{\HSaux_v}{\HSaux_e}$
  to a dense set).  As $\ran \Gamma = \bigoplus_{e \in E} \ran
  \Gamma_e$ (algebraic direct sum) with $\ran \Gamma_e=\ran
  \Gamma_{\bd_\pm e, e}$, the density of $\ran \Gamma$ in $\HSaux$
  follows.
  
  \itemref{vx-coupling.b}~We have $f \in \bigoplus_{v \in V} \ker
  \Gamma_v$ iff $f_v \in \ker \Gamma_v$ for all $v \in V$.  Moreover,
  $\ker \Gamma_v = \bigcap_{e \in E_v} \ker \Gamma_{v,e}$ as
  $\bigcap_{e \in E_v}\ker \pi_{v,e}=\{0\} \subset \HSaux_v$ (using
  the injectivity of $\iota_v$, see
  \Defenum{vx-coupling}{def.vx-coupling.b}).  By definition, $\Gamma_e
  := \Gamma_{v,e}$ for $v=\bd_\pm e$, hence we have
  \begin{equation*}
    \ker \Gamma = \bigoplus_{e \in E} \ker \Gamma_e = \bigoplus_{v \in
      V} \ker \Gamma_v.
  \end{equation*}

  \itemref{vx-coupling.c}~In particular, if all spaces $\ker \Gamma_v$
  are dense in $\HS_v$, then $\ker \Gamma$ is dense in $\HS$.

  \itemref{vx-coupling.c'}~follows from a simple calculation
  using~\eqref{eq:green} on each \aBVP $\Pi_v$.

  \itemref{vx-coupling.d} is obvious, as well
  as~\itemref{vx-coupling.e}

  \itemref{vx-coupling.f}~The formula follows from $\qf
  l_z(\phi,\psi)=(\qf h-z \qf 1)(S(z)\phi,S(-1)\psi)$ and
  part~\itemref{vx-coupling.d}.
\end{proof}

\begin{examples}
  \label{ex:disc-graph}
  \indent
  \begin{enumerate}
  \item
    In \Ex{vx-coupling}, the entire space is $\HS=\Lsqr X$ and the
    Neumann operator is the usual Laplacian on the graph-like manifold
    $X$.

  \item 
    \label{disc-graph2}
    Assume that $G=(V,E,\bd)$ is a discrete graph.  We decompose $G$
    into its \emph{star components} $G_v=(\{v\} \dcup E_v, E_v, \wt
    \bd)$, i.e., each edge in $G$ adjacent to $v$ becomes also a
    vertex in $G_v$.  As boundary of $G_v$ we set $\bd G_v=E_v$ If we
    identify the new vertices $e \in V(G_{\bd_-e})$ and $e \in
    V(G_{\bd_+e})$ of the star components $G_{\bd_-e}$ and $G_{\bd_+
      e}$ for all edges $e \in E$ we just obtain the subdivision graph
    $SG$ (see \Def{subdivision}).

    Let $\Pi_v$ be the \aBVP associated with the graph $G_v$ and
    boundary $\bd G_v=E_v$, i.e., $\Pi_v=(\C^{E_v},\Gamma_v, \qf h_v,
    \C^{E_v \dcup \{v\}},\C^{E_v \dcup \{v\}})$ with $\Gamma_v f= f
    \restr{E_v}$ (see \Exenum{abvp}{abvp-graph}), where
    $\C^{E_v}:=\set{\map \phi{E_v} \C}{\text{$\phi$ map}}$ denotes the
    set of maps or families with coordinates indexed by $e \in E_v$.
    Denote for short $d_v=\deg v$.  Of course, $\C^{E_v}$ is
    isomorphic to $\C^{d_v}$, but this isomorphism needs a numbering
    of the edges which is unimportant for our purposes.  The Neumann
    operator, written as a matrix with respect to the orthonormal
    basis $\phi_v:=d_v^{-1/2}\delta_v$ ($v \in V$), has block
    structure $A_v=\id_{E_v}$ (identity matrix of dimension $d_v$),
    $B_v=-d_v^{-1/2}(1, \dots, 1)^{\mathrm T}$ and $D_v=1$.  In particular, the
    \DtN operator is $\Lambda_v(z)=(1-z)\id_{E_v} - (d_v(1-z))^{-1}
    \1_{E_v \times E_v}$, where $\1_{E_v \times E_v}$ is the
    $(d_v\times d_v)$-matrix with all entries $1$.

    The vertex-coupled \aBVP $\Pi$ of the family $(\Pi_v)_{v \in V}$
    (it is clear that this family allows a vertex coupling) is now the
    \aBVP of the subdivision graph $SG$ of $G$ with boundary $\bd
    SG=E$, the edges of $G$.  More precisely, we have already
    identified the subspace $\set{\phi \in \bigoplus_{v \in V}
      \lsqr{G_v}}{f_{v,e}=f_{w,e} \;\forall e \in E, \bd e=(v,w)}$, with
    $\lsqr{SG}$.  The coupled Neumann operator is the Laplacian of the
    subdivision graph $SG$, i.e., $\HNeu=\laplacian{SG}$.  Note that
    we can embed $\lsqr G$ into $\lsqr{SG}$, $f \mapsto \wt f$, with
    $\wt f(v)=f(v)$ and $\wt f(e)=(f(\bd_+e)+f(\bd_-e))/2$; moreover,
    $2\qf h_G(f)=\qf h_{SG}(\wt f)$ for the corresponding energy forms.
    The coupled \DtN operator is
    \begin{equation*}
      \bigl(\Lambda(z) \phi)\bigr)_e 
      = 2(1-z) \phi_e 
      - \frac 1{1-z} \sum_{v =\bd_\pm e} \frac 1{\deg v}\sum_{e'\in E_v} \phi_{e'}.
    \end{equation*}
    For $z=0$, this is just the formula for a Laplacian on the line
    graph $LG$ of $G$ (the line graph has as vertices the edges of
    $G$, and two such edges are adjacent, if they meet in a common
    vertex, see e.g.~\cite[Ex.~3.14~(iv)]{post:09c}). In particular,
    if $G$ is $r$-regular, then $LG$ is $(2r-2)$-regular and
    \begin{equation}
      \label{eq:dtn.line-graph}
      \Lambda(z)
      = 2(1-z) - \frac{2r-2} {(1-z)r} (1-\laplacian{LG}).
    \end{equation}
  \end{enumerate}
\end{examples}
In particular, applying the spectral relation~\eqref{eq:spec-rel} to
the last example (the Dirichlet spectrum of all star components is
$\{1\}$ as $\HDir_v=D_v=1$) we rediscover the following result
(see~\cite{shirai:00}):
\begin{corollary}
  The spectra of the subdivision and line graph of an $r$-regular
  graph are related by
  \begin{equation*}
    \lambda \in \spec{\laplacian{SG}} \iff
    1- \frac r{r-1}(1-\lambda)^2 \in \spec{\laplacian{LG}}
  \end{equation*}
  provided $\lambda \ne 1$.
\end{corollary}

\subsection{Edge coupling}
\label{sec:ed-coupling}
Let us now couple \aBVPs indexed by the edges of a given graph
$G=(V,E,\bd)$: For each \emph{edge} $e \in E$ we assume that there is
an \aBVP $\Pi_e=(\Gamma_e,\HSaux_e,\qf h_e, \HS^1_e,\HS_e)$.

\begin{definition}
  \label{def:ed-coupling}
  We say that the family of \aBVPs $(\Pi_e)_{e \in E}$ allows an
  \emph{edge coupling}, if the following holds:
  \begin{enumerate}
  \item Assume that $\sup_{e \in E} \norm {\Gamma_e} < \infty$.
  \item Assume that for each vertex $e \in E$ there is a decomposition
    $\HSaux_e = \HSaux_{e,\bd_-e} \oplus \HSaux_{e,\bd_+e}$ and
    $\Gamma_e f = \Gamma_{e,\bd_-e} f \oplus \Gamma_{e,\bd_+e} f$,
    where $\map{\Gamma_{e,v}}{\HS^1_e}{\HSaux_{e,v}}$.
  \end{enumerate}
  We set $\HSaux_v^{\max} := \bigoplus_{e \in E_v} \HSaux_{e,v}$.
\end{definition}

Note that the sum over all maximal spaces $\HSaux_v^{\max}$ is the
decoupled space, as
\begin{equation}
  \label{eq:ed-dec}
  \HSaux^{\max}
  := \bigoplus_{v \in V} \HSaux_v^{\max} 
  = \bigoplus_{v \in V} \bigoplus_{e \in E_v} \HSaux_{e,v}
  = \bigoplus_{e \in E} \bigoplus_{v =\bd_\pm e} \HSaux_{e,v}
  =\bigoplus_{e \in E} \HSaux_e
  = \HSaux^\dec.
\end{equation}
Denote $\ul \phi(v):=(\phi_e(v))_{e \in E_v} \in \HSaux_v^{\max}$ the
collection of all edge contributions at the vertex $v \in V$, where
$\phi_e=(\phi_e(\bd_-e),\phi_e(\bd_+e)) \in \HSaux_{e,\bd_-e} \oplus
\HSaux_{e,\bd_+e}$.

Let $\HSaux_v \subset \HSaux_v^{\max}$ be a closed subspace for each
$v \in V$.  We construct an \aBVP $\Pi=(\Gamma,\HSaux,\qf h,
\HS^1,\HS)$ from the family $(\Pi_e)_e$ and the subspace
$\HSaux:=\bigoplus_v \HSaux_v$ as a restriction of the decoupled \aBVP
$\bigoplus_{e \in E} \Pi_e$ (see \Sec{dir-sum}):
\begin{align*}
  \HS:= &\bigoplus_{e \in E} \HS_e, \qquad %
  \HS^{1,\dec} := \bigoplus_{e \in E} \HS^1_e, \qquad %
  \qf h^\dec := \bigoplus \qf h_e;\\ %
  \HS^1 := & 
     \bigset{f=(f_e)_{e \in E} \in \HS^{1,\dec}} 
         {\Gamma^\dec f \in \HSaux}, \quad  %
  \qf h := \qf h^\dec \restr {\HS^1}, \quad %
  \Gamma := \Gamma^\dec \restr{\HS^1}
\end{align*}
where $\map{\Gamma^\dec }{\HS^{1,\dec}}{\HSaux^\dec=\HSaux^{\max}}$.
Denote by $\iota$ the embedding $\map \iota \HSaux {\HSaux^\dec}$.

\begin{theorem}
  \label{thm:ed-coupling}
  Assume that $(\Pi_e)_{e \in E}$ is a family of \aBVPs allowing an
  edge coupling and let $\HSaux_v \subset \HSaux_v^{\max}$ be a closed
  subspace for each $v \in V$, then the following holds:
  \begin{enumerate}
  \item
    \label{ed-coupling.a}
    The quintuple $\Pi=(\Gamma,\HSaux,\qf h, \HS^1,\HS)$ as
    constructed above is an \aBVP.
  \item
    \label{ed-coupling.b}
    We have $\ker \Gamma=\bigoplus_{e \in E} \ker \Gamma_e$, $\HDir =
    \bigoplus_{e \in E} \HDir_e$ (the Dirichlet operator is decoupled)
    and $\spec \HDir = \clo{\bigcup_{e \in E} \spec {\HDir_e}}$.
  \item
    \label{ed-coupling.c}
    In particular, if all \aBVPs $\Pi_e$ have dense Dirichlet domain
    then $\Pi$ also has dense Dirichlet domain.

  \item
    \label{ed-coupling.c'}
    The Neumann operator is coupled and is given by
    \begin{equation*}
      \dom \HNeu
      = \bigset{f \in \bigoplus_{e \in E} \WS_e} 
      {\Gamma f  \in \HSaux, \;
       \Gamma' f \in \HSaux^{\max} \ominus \HSaux}
    \end{equation*}
    with $\HNeu f= (\Hmax f)_{v \in V}$ (see~\eqref{eq:green} for the
    notation).

  \item 
    \label{ed-coupling.d}
    We have $S(z)\phi = S^\dec(z) \phi = \bigoplus_{e \in E}
    S_e(z)\phi_e$ where $\phi \in \HSaux \subset \HSaux^{\dec}$.

  \item
    \label{ed-coupling.e}
    Moreover, if all $\Pi_e$ are elliptically regular with uniformly
    bounded elliptic regularity constants (i.e., $\sup_{e \in
      E}\norm[\HSaux_e\to\HS_e]{S_e} < \infty$), then $\Pi$ is also
    elliptically regular.
  \item
    \label{ed-coupling.f}
    We have $\Lambda(z) = \iota^* \Lambda^\dec(z) \iota$. i.e., if
    $\psi=\Lambda(z)\phi$, then
    \begin{equation*}
      \ul \psi(v)
      := \bigl(\psi_e\bigr)_{e \in E_v}
      = \pi_v \bigl( (\Lambda_e(z) \phi_e)(v) \bigr)_{e \in E_v}
    \end{equation*}
    for $z \notin \spec \HDir$, where
    $\map{\pi_v}{\HSaux_v^{\max}}{\HSaux_v}$ is the adjoint of
    $\map{\iota_v}{\HSaux_v}{\HSaux_v^{\max}}$.
  \end{enumerate}
\end{theorem}
\begin{proof}
  The proof is very much as the proof of \Thm{vx-coupling}:
  \itemref{ed-coupling.a}~The operator $\Gamma^\dec$ is bounded, and
  $\HS^1$ is closed in $\HS^{1,\dec}$ as preimage of the closed
  subspace $\HSaux$ under $\Gamma^\dec$; in particular, $\Gamma$ is
  bounded.  Moreover, $\ran \Gamma=\Gamma(\HS^1)=\Gamma^\dec(\HS^1) =
  \HSaux \cap \Gamma^\dec(\HS^1)$; since $\Gamma^\dec(\HS^1)$ is dense
  (as all components $\Gamma_e(\HS^1_e)$ are dense in $\HSaux_e$, the
  space $\ran \Gamma$ is also dense in $\HSaux$.
   
  \itemref{ed-coupling.b}--\itemref{ed-coupling.f} can be seen
  similarly as in the proof of \Thm{vx-coupling}.
\end{proof}

\begin{examples}
  \label{ex:disc-graph-ed}
  Let us give some important examples of subspaces $\HSaux$ of
  $\HSaux^{\max}$:
  \begin{enumerate}
    \item
    \label{ed-coupling-2d}
    \myparagraph{Edge coupling of two-dimensional \aBVPs:} Assume that
    all vertex components $\HSaux_{e,v}$ equal $\C$.  Then
    $\HSaux_v^{\max}=\C^{E_v}$ and we choose for example $\HSaux_v :=
    \C(1,\dots,1)$ (\emph{standard} or \emph{Kirchhoff vertex
      conditions}), where $(1,\dots,1) \in \C^{E_v}$ has all $(\deg
    v)$-many components $1$.  It is convenient to choose $\abs[\deg v]
    w=\abs w \sqrt{\deg v}$ as norm on $\C$ (then $\C(1,\dots,1)
    \subset \C^{E_v}$ is isometrically embedded in $(\C,\abs[\deg v]
    \cdot)$ via $(\eta,\dots,\eta) \mapsto \eta$).  A vector $\ul
    \eta(v)$ is of course characterised by the common scalar value
    $\eta(v) \in \C$ and the projection $\pi_v \ul
    \psi(v)=(\psi_e(v))_{e \in E_v}$ is characterised by the sum
    $(\deg v)^{-1}\sum_{e \in E_v} \psi_e(v)$.  Hence we can write the
    \DtN operator as
      \begin{equation}
        \label{eq:dtn.ed-coupling}
        (\Lambda(z) \phi)(v)
        = \frac 1 {\deg v} \sum_{e \in E_v} (\Lambda_e(z) \phi_e)(v)
      \end{equation}
      for $\phi \in \HSaux = \lsqr{V,\deg}$.  This formula is a
      generalisation of the formula for the discrete (normalised)
      Laplacian.  One can also choose a more general subspace
      $\HSaux_v \subset \HSaux_v^{\max}=\C^{E_v}$, the resulting \DtN
      operators look like generalised discrete Laplacians described
      e.g.\ in~\cite[Sec.~3]{post.in:08} or~\cite[Sec.~3]{post:09c}.  A
      similar approach has been used in~\cite{pankrashkin:06}.
  \item
    \label{ed-coupling-2d-trivial}
    \myparagraph{Edge coupling of two-dimensional \emph{trivial} \aBVPs
      gives back the original discrete graph:} Let us now treat a
    special case of~\itemref{ed-coupling-2d}: Let $\Pi_e=(\id, \C^2,
    \qf h_e, \C^2,\C^2)$ be a trivial \aBVP for each $e \in E$.  The
    \aBVP $\Pi_v$ can be understood as coming from a graph consisting
    only of two vertices $\bd_\pm e$ and one edge $e$, and both
    vertices belong to the boundary.  The energy form is $\qf
    h_e(f)=\abssqr{f_2-f_1}$ for $f=(f_1,f_2)\in \C^2$, see
    \Exenum{abvp}{abvp-graph}).  In this case,
    \begin{equation*}
      \Lambda_e(z)
      =\begin{pmatrix}
         1-z & -1\\ -1 & 1-z
       \end{pmatrix}
    \end{equation*}
    and the \DtN operator becomes
      \begin{equation*}
        (\Lambda(z) \phi)(v)
        = \frac 1 {\deg v} \sum_{e \in E_v} (\phi(v)-\phi(v_e))
        - z \phi(v),
      \end{equation*}
      i.e., $\Lambda(z)=\laplacian G - z$, i.e., the \DtN operator is
      just the shifted normalised Laplacian on $G$.  Note that in this
      case, the Neumann Laplacian is also the \DtN operator at $z=0$,
      i.e., $\HNeu=\Lambda(0)=\laplacian G$, as the global form $\qf
      h$ is $\qf h(f)=\sum_e \qf h_e(f_e)$.
    \item 
      \label{synchronous}
      \myparagraph{Standard vertex conditions:} Here, we describe an
      edge coupling with a special choice of vertex spaces $\HSaux_v
      \subset \HSaux_v^{\max}$, similar to the standard or Kirchhoff
      vertex conditions of a quantum graph.

      Assume that for given $v \in V$, the vertex component
      $\HSaux_{e,v}$ of $\HSaux_e$ equals a given Hilbert space
      $\HSaux_{v,0}$ for all $e \in E_v$.  Then $\HSaux_v := \set{\ul
        \eta(v)=(\eta,\dots,\eta) \in \HSaux_v^{\max}}{\eta \in
        \HSaux_{v,0}}$ is a closed subspace (i.e., $\HSaux_v$ consists
      of the $(\deg v)$-fold diagonal of the model space
      $\HSaux_{v,0}$).  In the above two examples, we treated the case
      $\HSaux_{e,v}=\C$.  A vector $\ul \eta(v)$ is characterised by
      $\eta(v)$ and the projection $\pi_v \ul \psi(v)=(\psi_e(v))_{e
        \in E_v}$ is characterised by the sum $(\deg v)^{-1}\sum_{e
        \in E_v} \psi_e(v)$.  Hence we can write the \DtN operator
      exactly as in~\eqref{eq:dtn.ed-coupling}.  This formula is a
      vector-valued version of a normalised discrete Laplacian
      (see~\eqref{eq:dtn.vec.q-graph} for a more concrete formula).
  \end{enumerate}
\end{examples}

\subsection{Vector-valued quantum graphs}
\label{sec:vect-q-graphs}

Let $I=[a,b]$ be a compact interval of length $\ell=b-a>0$, and let
$\anHS$ be a Hilbert space with non-negative closed quadratic form
$\qf k \ge 0$ such that its corresponding operator has purely discrete
spectrum.  We set $\HS := \Lsqr{I,\anHS} \cong \Lsqr I \otimes \anHS$
and define an energy form via
\begin{equation*}
  \qf h(f) 
  := \int_I \bigl(\normsqr[\anHS]{f'(t)} + \qf k(f(t)) \bigr)\dd t
\end{equation*}
for $f \in \Lsqr{I,\anHS}$ such that $f \in \Cont[1]{I,\anHS}$ and
$f(t) \in \dom \qf k$ for almost all $t \in I$.  Denote by $\qf h$
also the closure of this form.  As boundary space set $\HSaux := \anHS
\oplus \anHS \cong \anHS \otimes \C^2$ and define $\Gamma f :=
(f(a),f(b))$.  It is not hard to see that $\Pi=(\Gamma,\HSaux,\qf h,
\dom \qf h, \HS)$ is an elliptically regular \aBVP with dense
Dirichlet domain; moreover, the norm of $\Gamma$ is bounded by
$\sqrt{\coth(\ell/2)}$ (see e.g.~\cite[Sec.~6.1 and~6.4]{post:15}).
We call $\Pi$ the \emph{\aBVP associated with $(\qf k,\anHS,I)$}.
Moreover, as the underlying space of $\Pi$ has a product structure, we
can calculate all derived objects explicitly.  For example, the \DtN
operator $\Lambda(z)$ is an operator function of a matrix with respect
to the decomposition $\HSaux =\anHS \oplus \anHS$.  In particular, we
have $\Lambda(z) = \Lambda_0(z-K)$, where $K$ is the operator
associated with $\qf k$ and where
\begin{equation}
  \label{eq:dtn.matrix}
  \Lambda_0(z)
  = \frac {\sqrt z}{\sin(\ell \sqrt z)}
  \begin{pmatrix} 
    \cos(\ell \sqrt z) & -1\\
    -1 & \cos(\ell \sqrt z)
  \end{pmatrix}
\end{equation}
is the \DtN operator for the scalar problem ($\anHS=\C$)
\begin{equation*}
  \Pi_0=\bigl(\Gamma_0,\C^2,\qf h_0,\Sob I,\Lsqr I\bigr)
\end{equation*}
with $\Gamma_0f=(f(a),f(b))$ and $\qf h_0(f)=\int_I \abssqr{f'(t)} \dd
t$ (see e.g.~\cite{behrndt-post:pre16} for details).  The complex
square root is cut along the positive real axis.  The same argument
also works for abstract warped products.

Let us now consider vector-valued quantum graphs: Assume that
$G=(V,E,\bd)$ is a discrete graph and that $I_e$ is a closed interval
of length $\ell_e$ for each $e \in E$.  Assume that there is a Hilbert
space $\anHS_e$ and energy form $\qf k_e$ for each edge $e \in E$.
Then we can edge-couple the family of \aBVP $\Pi_e$ associated with
$(\qf k_e, \anHS_e, I_e$).  In order to formulate the next result, we
define the \emph{unoriented} and \emph{oriented evaluation} of $f$ at
a vertex $v$ and edge $e$ by
\begin{equation*}
  f_e(v) 
  =
  \begin{cases}
    f_e(\min I_e),& v=\bd_-e\\
    f_e(\max I_e),& v=\bd_+e
  \end{cases}
  \quadtext{and}
  \orient f_e(v) 
  =
  \begin{cases}
    -f_e(\min I_e),& v=\bd_-e\\
    +f_e(\max I_e),& v=\bd_+e
  \end{cases}
\end{equation*}
for $f \in \bigoplus_{e \in E} \Sob{I_e,\anHS_e}$.
\begin{theorem}
  \label{thm:vect.q-graph}
  Let $\Pi_e=(\Gamma_e,\anHS_e \oplus \anHS_e,\qf h_e, \dom \qf h_e,
  \Lsqr{I_e,\anHS_e})$ be an \aBVPs associated with $(\qf k_e,
  \anHS_e, I_e)$.  Assume that the length $\ell_e$ of $I_e$ fulfils
    \begin{equation}
      \label{eq:lower.bd}
      0 < \ell_0 := \inf_{e \in E}\ell_e.
    \end{equation}
    \begin{subequations}
      \begin{enumerate}
      \item
        \label{vect.q-graph.a}
        Then the family $(\Pi_e)_{e \in E}$ allows an edge coupling.
        As boundary space we choose $\HSaux=\bigoplus_{v \in
          V}\HSaux_v$, where $\HSaux_v$ is a closed subspace of
        $\HSaux_v^{\max}:=\bigoplus_{e \in E_v}\anHS_e$ for each $v \in
        V$.  Then $H$ acts as
        \begin{equation}
          \label{eq:vect.q-graph.op}
          (Hf)_e(t) =-f_e''(t) + K_e f(t)
        \end{equation}
        on each edge, where $K_e$ is the operator associated with $\qf
        k_e$.  Moreover, a function $f$ in the domain of the Neumann
        operator $\HNeu$ of the edge-coupled \aBVP fulfils
        \begin{equation}
          \label{eq:vect.q-graph.dom}
          \ul f(v) 
          =\bigl(f_e(v)\bigr)_{e \in E_v} \in \HSaux_v
          \quadtext{and}
          \orul f'(v)
          =\bigl(\orient f_e(v)\bigr)_{e \in E_v} \in 
          \HSaux_v^{\max} \ominus \HSaux_v.
        \end{equation}
      \item
        \label{vect.q-graph.b}
        If $H$ is a self-adjoint operator in $\HS=\bigoplus_{e \in
          E}\Lsqr{I_e,\anHS_e}$ such that~\eqref{eq:vect.q-graph.op}
        holds for all functions $f=(f_e)_e \in \dom H$ such that $f_e
        \in \Cont[2]{I_e,\dom K_e}$ vanishing near $\bd I_e$, and such
        that the values $\ul f(v)$ and $\orul f'(v)$ are not coupled
        in $\dom H$, then there exist closed subspaces $\HSaux_v$ of
        $\HSaux_v^{\max}$ for each $v \in V$, such that $\dom H$ is of
        the form~\eqref{eq:vect.q-graph.dom}
      \end{enumerate}
    \end{subequations}
\end{theorem}
We call $\HNeu$ the \emph{vector-valued quantum graph Laplacian with vertex spaces $\HSaux_v$ and fibre operators $K_e$}.
\begin{proof}
  Part~\itemref{vect.q-graph.a} follows already from the discussion
  above, the fact that $\Gamma f = (\ul f(v))_{v \in V}$ and $\Gamma'
  f = (\orul f'(v))_{v \in V}$, and \Thm{ed-coupling}.  Note that
  $\normsqr{\Gamma_e}$ is bounded by $2/\min\{\ell_0,1\}$, see
  \cite[Cor.~A.2.12]{post:12}.

  For part~\itemref{vect.q-graph.b}, partial integration shows that
  \begin{align*}
    \iprod[\HS]{Hf} g
    &= \sum_{e \in E} \int_{I_e} 
    \iprod[\anHS_e]{-f_e''(t)+ K_ef_e(t)}{g_e(t)} \dd t\\
    &= \sum_{e \in E} 
    \Bigl(
      \int_{I_e} 
         \bigl( \iprod[\anHS_e]{f'_e(t)}{g_e'(t)}
              + \qf k_e(f_e(t),g_e(t)) 
         \bigr) 
      \dd t
      - \Bigl[\iprod[\anHS_e]{f_e'(t)}{g_e(t)} \Bigr]_{\bd I_e}
    \Bigr)\\
    &= \iprod[\HS]f {Hg}
     + \sum_{e \in E} 
     \Bigl[ \iprod[\anHS_e]{f_e(t)}{g_e'(t)}
           -\iprod[\anHS_e]{f_e'(t)}{g_e(t)}  \Bigr]_{\bd I_e}
  \end{align*}
  for $f,g \in \dom H$.  Reordering the boundary contributions (the
  last sum over $e \in E$) gives
  \begin{align*}
    \sum_{v \in V} \sum_{e \in E_v} 
      \bigl(
       \iprod[\anHS_e]{\orient f_e'(v)}{g_e(v)}
      -\iprod[\anHS_e]{f_e(v)}{\orient g_e'(v)}
      \bigr)\\
    =\sum_{v \in V} 
      \bigl(
        \iprod[\HSaux_v^{\max}]{\orul f'(v)}{\ul g(v)}
       -\iprod[\HSaux_v^{\max}]{\ul f(v)}{\orul g'(v)}
      \bigr),
  \end{align*}
  As the values $\orul f'(v)$ and $\ul f(v)$ resp.\ $\orul g'(v)$ and
  $\ul g(v)$ are not coupled, each contribution
  $\iprod[\HSaux_v^{\max}]{\orul f'(v)}{\ul g(v)}$ and
  $\iprod[\HSaux_v^{\max}]{\ul f(v)}{\orul g'(v)}$ has to vanish
  separately.  We let $\HSaux_v$ be the closure of the linear span of
  all boundary values $\ul f(v)$, $f \in \dom H$.  In particular, we
  then have $\orul f'(v), \orul g'(v) \in \HSaux_v^\perp$.
\end{proof}
\begin{remarks}
  \indent
  \begin{enumerate}
    \item For simplicity, we describe only the \emph{energy
      independent} vertex conditions, not involving any condition
      between the values $\ul f(v)$ and $\orul f'(v)$.  One can, of
      course, also consider Robin-type conditions, but one needs
      additional finiteness or boundedness conditions in this case.
    \item If $\anHS_e=\C$ for all edges $e \in E$, then we have
      defined an ordinary quantum graph.  For the case $(\qf k_e,
      \anHS_e, I_e)=(0,\anHS_0,[0,1])$ for all $e \in E$, where
      $\anHS_0$ is a given Hilbert space,
      see~\cite{von-below-mugnolo:13}.
  \end{enumerate}
\end{remarks}

We have not used the whole power of \aBVPs here, namely the resolvent
formula~\eqref{eq:krein-res}.  In this setting, the left hand side,
the resolvent of $\HNeu$ in $z \notin \spec \HNeu \cup \spec \HDir$,
equals the right hand side, which can be expressed completely in terms
of the building blocks $\Pi_e$.  Moreover, the \DtN operator has the
nature of a discrete Laplacian.

If we use standard vertex conditions (see
\Exenum{disc-graph-ed}{synchronous}), we have to assume that all
boundary spaces $\anHS_e$ are the same and equal (or can at least be
naturally identified with) a Hilbert space $\anHS_0$.  Then we set
$\HSaux_v=\anHS_0(1,\dots,1)$, i.e., all $\deg v$ components of $\ul
\phi(v) \in \HSaux_v$ are the same.  In this case, the \DtN operator
is (see~\eqref{eq:dtn.ed-coupling} and~\eqref{eq:dtn.matrix})
\begin{equation}
  \label{eq:dtn.vec.q-graph}
  (\Lambda(z) \phi)(v)
  = \frac 1 {\deg v} \sum_{e \in E_v} 
  \bigl(C_e(z)\phi(v)-S_e(z)\phi(v_e)\bigr)
\end{equation}
where $C_e(z)=\sqrt{z-K_e}\cot(\ell_e \sqrt{z-K_e})$ and
$S_e(z)=\sqrt{z-K_e}/\sin(\ell_e \sqrt{z-K_e})$.

\myparagraph{The equilateral and standard (or \emph{Kirchhoff}) case.}
Let us now characterise the spectrum of the vector-valued quantum
graph Laplacian $\HNeu$ in a special case:
\begin{example} 
  If all \aBVPs $\Pi_e$ are the same (or isomorphic) and all lengths
  $\ell_e$ are the same (say, $\ell_e=1$), then we call the
  vector-valued quantum graph \emph{equilateral} and $K_e=K_0$ on
  $\anHS_e=\anHS_0$ for all $e \in E$).  A related case ($K_0=0$) has
  also been treated in~\cite{von-below-mugnolo:13}.  In the
  equilateral case, we have
  \begin{equation}
    \label{eq:dtn.equilateral}
    \Lambda(z)
    =
    \bigl(1 \otimes (1/\sin\sqrt{z-K_0})\bigr)
    \Bigl(
    1 \otimes \cos\sqrt{z-K_0}-1 
    + \laplacian G \otimes 1
    \Bigr)
  \end{equation}
  where $\laplacian G$ denotes the (discrete) normalised Laplacian
  (see~\eqref{eq:normalised.lapl}) and where we have identified
  $\HSaux \cong \lsqr {V,\deg} \otimes \anHS_0$.  Since $1 \otimes
  (1/\sin\sqrt{z-K_0})$ is a bijective operator and $\spec{A \otimes
    1-1 \otimes B}=\spec A-\spec B=\set{a-b}{a\in \spec A, b \in \spec
    B}$, we have in particular (using~\eqref{eq:spec-rel} for the
  first equivalence)
  \begin{align*}
    \lambda \in \spec \HNeu
    \iff
    0 \in \spec {\Lambda(z)}
    &\iff
    0 \in \bigspec {1-\cos\sqrt{z-K_0}}-\spec{\laplacian G}\\ 
    &\iff
    \bigspec {1-\cos\sqrt{z-K_0}} 
    \cap \spec{\laplacian G} \ne \emptyset\\
    &\iff
    \exists \kappa \in \spec {K_0} \colon\;
    1-\cos\sqrt{z-\kappa} \in \spec{\laplacian G}
  \end{align*}
  provided $z \notin \spec \HDir= \spec{\HDir_{[0,1]} \otimes 1 +
    1 \otimes K_e} = \set{(n\pi)^2+\kappa}{n=1,2,\dots, \kappa \in
    \spec {K_0}}$.
\end{example}

We have therefore shown the following:
\begin{corollary}
  \label{cor:equilateral}
  Assume that all edge \aBVPs $\Pi_e$ are the same, i.e., associated
  with $(\qf k_0, \anHS_0, [0,1])$ (see the beginning of
  \Sec{vect-q-graphs}), and that all vertex spaces are standard
  ($\HSaux_v=\set{(\phi,\dots,\phi) \in \anHS_0^{\deg v}}{ \phi\in
    \anHS_0}$).  Then the \DtN operator is given
  by~\eqref{eq:dtn.equilateral}.  Moreover, the spectrum of the
  vector-valued quantum graph Laplacian $\HNeu$ is characterised by
  \begin{equation*}
    \lambda \in \spec \HNeu
    \quad\iff\quad
    \exists \kappa \in \spec {K_0} \colon\;
    1-\cos\sqrt{\lambda-\kappa} \in \spec{\laplacian G}
  \end{equation*}
  provided $\lambda \notin \spec
  \HDir=\set{(n\pi)^2+\kappa}{n=1,2,\dots; \kappa \in \spec {K_0}}$.
\end{corollary}
Molchanov and Vainberg~\cite{molchanov-vainberg:06} treated the
asymptotic behaviour $\eps \to 0$ of a Dirichlet Laplacian on the
product $X \times (M,g_\eps)$, where $X$ is a metric graph, $(M,g)$ is
a compact Riemannian manifold with boundary and $g_\eps=\eps^2g$.  In
our notation, it means that $\anHS_0=\Lsqr{M,g_\eps}$ and
$K_0=\laplacian{(M,g_\eps)}=\eps^{-2}\laplacian{(M,g)}$ with Dirichlet
boundary conditions.  It would be interesting to compare this model
with the usual $\eps$-tubular neighbourhood model with Dirichlet
boundary conditions.  Our methods allow such an analysis, see
\Sec{conv-graph-like}.

\subsection{Vertex-edge coupling}
\label{sec:vx-ed-coupling}
Here, we treat the coupling when there are building blocks for each
vertex-edge of a given graph $G=(V,E,\bd)$.  Formally this
coupling is just a \emph{vertex-based coupling} for the corresponding
\emph{subdivision graph} $SG=(A,B,\wt \bd)$ of $G$ (see
\Def{subdivision}).  Assume that $\Pi_a$ is an \aBVP for each vertex
$a \in A=V \dcup E$ of the subdivision graph.  The family $(\Pi_a)_{a
  \in A}$ allows a vertex-edge coupling if the following holds:

We say that the family of \aBVPs $(\Pi_v)_{v \in V}$ allows a
\emph{vertex-edge coupling}, if the following holds:
\begin{enumerate}
\item We assume that $\sup_{a \in A} \norm{\Gamma_a} < \infty$.

\item Assume there is a Hilbert space $\HSaux_{e,v}$ for each edge $v
  \in V$ and $e \in E_v$ (i.e., each edge $b=(v,e)$ of the subdivision
  graph).
  
  For the vertex vertices of $SG$ assume that there is a bounded
  operator $\map{\pi_{v,e}}{\HSaux_v}{\HSaux_{e,v}}$.  We set
  $\map{\Gamma_{v,e}:=\pi_{v,e} \Gamma_v}{\HS^1_v}
  {\HSaux_{e,v}}$.  Moreover, let $\HSaux_v^{\max}:=
  \bigoplus_{e \in E_v} \HSaux_e$ and
  $\map{\iota_v}{\HSaux_v}{\HSaux_v^{\max}}$, $\iota_v
  \phi_v=(\pi_{v,e} \phi_v)_{e \in E_v}$.  We assume that $\iota_v$ is
  an isometric embedding.

  For the edge vertices of $SG$ we assume that $\HSaux_e =
  \bigoplus_{v=\bd_\pm e} \HSaux_{e,v}$ (i.e., the vertex coupling is
  maximal at $e \in A$).  We set
  $\map{\Gamma_{e,v}}{\HS^1_e}{\HSaux_{e,v}}$, $\Gamma_{e,v} f_e
  := (\Gamma_e \phi_e)_v$.

\item For each edge $(v,e) \in B$, assume that $\ran \Gamma_{v,e}=\ran
  \Gamma_{e,v}=: \HSaux_{e,v}^{1/2}$.
\end{enumerate}

The formulas for the (subdivision) vertex-coupled \aBVP can be taken
from \Sec{vx-coupling} verbatim.  Let us give two typical examples of
such couplings:

\myparagraph{Vertex-edge coupling with maximal coupling space:
  shrinking graph-like manifolds.}
Consider a thin neighbourhood $X=X_\eps$ of an embedded metric graph
$X_0$ or a thin graph-like manifold of dimension $d \ge 2$ and
decompose $X=X_\eps$ into its closed vertex and edge neighbourhoods
$X_v=X_{\eps,v}$ and $X_e=X_{\eps,e}$, respectively (see
e.g.~\cite{exner-post:05,exner-post:09} or~\cite{post:12}).  We omit
the shrinking parameter $\eps>0$ in the sequel, as it does not affect
the coupling.

The \aBVPs $\Pi_v$ and $\Pi_e$ are the ones for $X_v$ and $X_e$ with
(internal) boundary $\bd X_v$ and $\bd X_e$ given as follows: Let
$Y_{e,v}:= X_e \cap X_v$ and assume that $Y_{e,v}$ is isometric with a
smooth $(d-1)$-dimensional manifold $Y_e$ for $v=\bd_\pm e$.  The
(internal) boundary of $X_v$ and $X_e$ is now $\bd X_v=\bigcup_{e \in
  E_v} (X_v \cap X_e)$ and $\bd X_e = \bigcup_{v = \bd_\pm e} (X_v
\cap X_e)$, respectively.  We also assume that $X_e$ is a product $I_e
\times Y_e$ with $I_e=[0,\ell_e]$.

The boundary spaces for each edge $b=(v,e) \in B$ are
$\HSaux_{e,v}=\Lsqr{Y_{e,v}} \cong \Lsqr {Y_e}$.  Note that
$\HSaux_v = \bigoplus_{e \in E_v} \HSaux_{e,v}$, i.e., that the vertex
coupling at the (original) vertices $v \in A$ is also maximal.

Under the typical uniform lower positive bound~\eqref{eq:lower.bd} and
a suitable decomposition of $X$ into $X_v$ and $X_e$, one can show
that $\norm {\Gamma_a}$ is uniformly bounded.  The other conditions
above can also be seen easily.

If we consider again the shrinking parameter and if we assume that
$X_{\eps,v}$ is $\eps$-homothetic (see example before
\Prp{triv.conv}), then (depending on the energy form and boundary
conditions chosen on $X$ if the external boundary $\bd X$ is
non-trivial), we can apply \Prp{triv.conv}.  \emph{Note that the
  boundary values of the eigenfunction on $X_v$ corresponding to the
  eigenvalue $0$ determine the limit boundary space, i.e., the vertex
  coupling appearing below as a proper subspace of $\HSaux_v^{\max}$.}

\myparagraph{Vertex-edge coupling with non-maximal coupling
  space: trivial vertex \aBVPs.}

Let us now construct another vertex-edge coupled \aBVP appearing e.g.\
in the limit situation of a shrinking graph-like space:

Assume that $(\Pi_e)_{e \in E}$ allows an edge coupling and that each
\aBVP $\Pi_e$ is bounded.  For each $v \in V$ choose a subspace
$\HSaux_v$ of $\HSaux_v^{\max}$ (e.g.\ given as the boundary values of
the eigenfunction corresponding to $0$ of a vertex neighbourhood).
For the vertex \aBVPs assume that $\Pi_v = \bigl(\id, \HSaux_v,
0,\HSaux_v, \HSaux_v \bigr)$, i.e., all $\Pi_v$'s are trivial (see
\Def{bd2}).

The corresponding vertex-edge-coupled \aBVP $\wt \Pi=(\wt
\Gamma,\HSaux,\wt{\qf h}, \wt \HS^1,\wt \HS)$ is given as follows: The
coupling condition in $\wt \HS^1$ here reads as
$\Gamma_{e,v}f_e=\Gamma_{v,e} f_v$ for all $e \in V_v$ and $v \in V$.
We define $\Gamma_v^\intl f := (\Gamma_{v,e} f_v)_{e \in E_v}$ and
$\Gamma_v^\extl f := (\Gamma_{e,v} f_e)_{e \in E_v}$, hence the
coupling condition becomes $\Gamma_v^\intl f = \Gamma_v^\extl f$.  As
$\Gamma_v^\intl f_v =(\Gamma_{v,e} f_v)_{e \in E_v}=\Gamma_v f_v=f_v
\in \HSaux_v$ the coupling condition is equivalent with
$f_v=\Gamma_v^\extl f \in \HSaux_v$.  Hence we have
\begin{align*}
  \wt \HS&=\bigoplus_{e \in E} \HS_e \oplus \bigoplus_{v \in V} \HSaux_v,
  \qquad
  \HSaux = \bigoplus_{v \in V} \HSaux_v,\\
  \wt \HS^1&=\bigset{f=(f_a)_{a \in E \dcup V} \in \wt \HS^{1,\dec}}
     {\Gamma_v^\extl f =f_v \in \HSaux_v \; \forall v \in V}\\
     \wt{\qf h} &= \wt{\qf h}^{\dec} \restr {\HS^1},\; \text{i.e.},\; 
     \wt{\qf h}(f)=\sum_{e \in E}\qf h_e(f_e), \qquad
     \wt \Gamma f = (\Gamma_v^{\extl} f)_{v \in V}
\end{align*}
(the $(\cdot)^\dec$-labelled objects are defined as in \Sec{dir-sum}).
Such operators have been treated in~\cite[Sec.~3.4.4]{post:12} under
the name ``\emph{extended operator}''.  

We have the following result, showing that the vertex-edge
coupling with trivial vertex \aBVPs leads just to the edge coupling:
\begin{theorem}
  \label{thm:vx-ed-non-max}
  Let $G$ be a discrete graph.  Assume that $\Pi$ is an \aBVP obtained
  from a family $(\Pi_e)_{e \in E}$ of \aBVPs allowing an edge
  coupling and choose a closed subspace $\HSaux_v$ of
  $\HSaux_v^{\max}=\bigoplus_{e \in E_v} \HSaux_{e,v}$ for each
  vertex.  As vertex family $(\Pi_v)_{v \in V}$ we choose the trivial
  \aBVP $\Pi_v=(\id,\HSaux_v,0,\HSaux_v,\HSaux_v)$ for each vertex.

  Then the vertex-edge-coupled \aBVP $\wt \Pi$ (i.e., vertex-coupled
  with respect to the subdivision graph $SG$) is equivalent (see
  below) with the edge-coupled \aBVP $\Pi$ according to the original
  graph $G$.
\end{theorem}
\begin{proof}
  Equivalence of two \aBVPs $\Pi$ and $\wt \Pi$ means that there are
  bicontinuous isomorphisms $\map{U^1}{\HS^1}{\wt \HS^1}$ and $\map T
  \HSaux {\wt \HSaux}$ such that $T \Gamma = \wt \Gamma U^1$ and
  $\wt {\qf h}(U^1f)=\qf h(f)$.  Here we have
  \begin{equation*}
    \HS=\bigoplus_{e \in E}\HS_e, \qquad
    \HS^1=\bigset{f \in \bigoplus_{e \in E} \HS^1_e}
      {\forall v \in V \colon \; \Gamma_v^\extl f \in \HSaux_v}
  \end{equation*}
  and $\wt \HS^1$ is given above. Set $U^1f=(f,(\Gamma_v^\extl f)_{v
    \in V})=:(f,\Gamma^\extl f)$ then $\wt {\qf h}(U^1f)=\qf h(f)$ and
  \begin{equation*}
    \normsqr[\wt \HS^1]{U^1f}
    = \wt{\qf h}(U^1f) + \normsqr[\HS] f 
     + \normsqr[\HSaux] {\Gamma^\extl f}
    = \normsqr[\HS^1] f + \normsqr[\HSaux] {\Gamma^\extl f}
    \le (1+\norm {\Gamma^\extl}) \normsqr[\HS^1] f
  \end{equation*}
  while for the inverse $(U^1)^{-1}(f,f_0)=f$ (with $f \in \HS$, $f_0
  \in \HSaux$) we have the estimate
  $\normsqr[\HS^1]{(U^1)^{-1}(f,f_0)} = \normsqr[\HS^1] f \le
  \normsqr[\wt \HS^1]{(f,f_0)}$.  Moreover, we set $T\phi=\phi$, then
  $T\Gamma f=\Gamma f=\wt \Gamma (U^1f)$.
\end{proof}
The previous result allows us to consider convergence of vertex-edge
coupled \aBVPs component-wise, even if the limit problem is only
edge-coupled.  A typical example is the convergence of the Laplacian
on a thin $\eps$-neighbourhood of a metric graph, to a Laplacian on
the metric graph.  We discuss the general convergence scheme in the
next section.

%
\section{Convergence of abstract graph-like spaces}
\label{sec:conv-graph-like}
%

In this section we show how one can translate convergence of building
blocks into a global convergence, expressed via the coupling of
abstract graph-like spaces in \Sec{graph-like} and the concept of
quasi-unitary equivalence resp.\ quasi-isomorphy for \aBVPs acting in
different Hilbert spaces in \Sec{conv-bd2}.

Fix a graph $G=(V,E,\bd)$ and let $\Pi$ and $\wt \Pi$ be two
vertex-coupled \aBVPs arising from the building blocks $\Pi_v$ and
$\wt \Pi_v$.  As we have seen in \Sec{vx-ed-coupling}, the vertex
coupling comprises also the vertex-edge coupled and even some
edge-coupled cases.

We want to show the following: Assume that all building blocks $\Pi_v$
and $\wt \Pi_v$ are quasi-isomorphic then the coupled Neumann forms
$\qf h$ and $\wt{\qf h}$ of the vertex-coupled \aBVPs $\Pi$ and $\wt
\Pi$ are quasi-unitarily equivalent, as well as the coupled boundary
operators $\Gamma$ and $\wt \Gamma$ are quasi-isomorphic.  We will not
treat the full (natural) problem of the quasi-isomorphy of $\Pi$ and
$\wt \Pi$ here, as we will not show that the boundary identification
operators $I$, $I'$ are quasi-isomomorphic (this would mean to impose
additional assumptions).

One problem with the coupling is that the naively defined
identification operator acts as
\begin{equation*}
  \map{J^{1,\dec} := \bigoplus_{v \in V} J^1_v}
  {\HS^{1,\dec}=\bigoplus_{v \in V} \HS^1_v}
  {\wt \HS^{1,\dec}=\bigoplus_{v \in V} \wt \HS^1_v}
\end{equation*}
but it is a priori not true that $J^{1,\dec}(\HS^1)\subset \wt \HS^1$,
i.e., that $J^{1,\dec}$ respects the coupling condition along the
different vertex building blocks as in \Sec{vx-coupling}.  In order to
correct this, we need the following definition (for the existence of
such operators, see the propositions after our next theorem):
\begin{definition}
  \label{def:smoothing}
  Let $\Pi$ be a vertex-coupled \aBVP arising from the vertex building
  blocks $(\Pi_v)_{v \in V}$.  We say that $\Pi$ \emph{allows a
    smoothing operator} if there is a bounded operator $\map
  B{\HS^{1,\dec}}{\HS^{1,\dec}}$ such that $f-Bf \in \HS^1$ for all $f
  \in \HS^{1,\dec}$.
\end{definition}
A simpler version of the following result can also be found
in~\cite[Sec.~4.8]{post:12}:
\begin{theorem}
  \label{thm:vx-coupling.que}
  Let $G=(V,E,\bd)$ be a discrete graph and $(\Pi_v)_v$, $(\wt
  \Pi_v)_v$ two families of \aBVPs allowing a vertex coupling. Assume
  that $\Pi_v$ and $\wt \Pi_v$ are $\delta_v$-quasi-isomorphic for
  each $v \in V$.  Moreover, assume that $\delta:=\sup_{v \in
    V} \delta_v <\infty$ and that the vertex-coupled \aBVPs $\Pi$ and
  $\wt \Pi$ allow smoothing operators $B$ and $\wt B$ such that
  \begin{equation}
    \label{eq:smoothing-op-small}
    \norm[\wt \HS^{1,\dec}]{\wt B J^{1,\dec} f}\le \delta \norm[\HS^{1,\dec}] f
    \quadtext{and}
    \norm[\HS^{1,\dec}]{B J^{\prime 1,\dec} u}\le \delta \norm[\wt \HS^{1,\dec}] u
  \end{equation}
  for $f \in \HS^1$ and $u \in \wt\HS^1$.  Then $\qf h$ and $\wt {\qf
    h}$ are $\delta'$-quasi-unitarily equivalent; and $\Gamma$ and
  $\wt \Gamma$ are $\delta'$-close where $\delta'=\Err(\delta)$.
\end{theorem}
\begin{proof}
  We define $\map{J^1}{\HS^1}{\wt \HS^1}$, $J^1 := (\id_{\wt
    \HS^1}-\wt B)J^{1,\dec}$.  From the smoothing property, we have
  $J^1f \in \wt \HS^1$ for any $f \in \HS^{1,\dec}$, hence $J^1$ maps
  into the right space.  Similarly, we define $J^{\prime 1} :=
  (\id_{\HS^1}-B)J^{\prime 1,\dec}$.  The identification operators on
  $\HS$ and $\wt \HS$ are given as $J:=\bigoplus_{v \in V} J_v$ and
  $J':=\bigoplus_{v \in V} J'_v$.  Then we have
  \begin{equation*}
    \normsqr[\wt \HS] {Jf-J^1f}
    \le 2\sum_{v \in V}\normsqr[\wt \HS_v]{J_vf_v - J^1_vf_v} 
    + 2\normsqr[\wt \HS^{1,\dec}] {\wt B J^{1,\dec} f}
    \le 4 \delta^2\normsqr[\HS^1] f
  \end{equation*}
  using our assumptions.  A similar property holds for $J'$ and
  $J^{\prime 1}$.  The other properties of \Def{id-ops-q-u-e} for $J$
  and $J'$ follow directly from the ones of $J_v$ and $J'_v$ (as
  in~\cite[Sec.~4.8]{post:12}).  For the $\delta$-closeness of $\qf h$
  and $\wt{\qf h}$ we have
  \begin{multline*}
    \bigabssqr{\wt{\qf h}(J^1 f, u) - \qf h(f, J^{\prime 1}u)}
    \le 3\Bigl(\sum_{v \in V} \bigabs{\wt{\qf h_v}(J^1_v f_v, u_v) 
                               - \qf h_v(f_v, J^{\prime 1}_vu_v)}\Bigr)^2\\
      + 3 \bigabssqr{\wt {\qf h}(\wt BJ^{1,\dec}f,u)}
      + 3 \bigabssqr{\qf h(f, BJ^{\prime 1,\dec}u)}
    \le 9\delta^2 \normsqr[\HS^1] f \normsqr[\wt \HS^1] u
  \end{multline*}
  using again~\eqref{eq:smoothing-op-small}.  For the boundary
  identification operators we set $\map I \HSaux {\wt \HSaux}$ where
  $(I\phi)_e =\frac 12\sum_{v = \bd_\pm e} \wt \pi_{v,e} I_v
  \pi_{v,e}^* \phi_e$
  and similarly for $I'$.  Then, we have the following estimates for
  the closeness of the boundary maps
    \allowdisplaybreaks
  \begin{align*}
    \normsqr[\wt \HSaux]{(I\Gamma - \wt \Gamma J^1)f}
    &= \sum_{e \in E}
    \normsqr[\wt \HSaux_e]{((I\Gamma - \wt \Gamma J^1)f)_e}\\
    &\le \sum_{e \in E} \sum_{v =\bd_\pm e}
    \frac 12 \bignormsqr[\wt \HSaux_e]
      {\wt \pi_{v,e}\bigl(I_v \Gamma_v f_v - (\wt \Gamma J^1f)_v \bigr)}\\
    &\le \frac 12\sum_{v \in V}
    \bignormsqr[\wt \HSaux_v^{\max}] 
               {\wt \iota_v(I_v \Gamma_v f_v - (\wt \Gamma J^1 f)_v)}\\
    &\le \sum_{v \in V}
    \bignormsqr[\wt \HSaux_v] 
               {(I_v \Gamma_v - \wt \Gamma J^1_v) f_v}
     + \sup_v \normsqr{\wt \Gamma_v} 
     \normsqr[\wt \HS^1]{\wt B J^{1,\dec} f}\\
    & 
    \le \delta^2\bigl(1+\sup_v \normsqr{\wt \Gamma_v}\bigr) \normsqr[\HS^1] f.
  \end{align*}
  by our assumptions.  Similarly, we show the related property for
  $I'\wt \Gamma-\Gamma J^{\prime 1}$.

  $\delta'=\delta\max\{3,\sup_v\norm{\Gamma_v}+1,\sup_v \norm{\wt
    \Gamma_v}+1\}$ will do the job.
\end{proof}

Let us now prove the existence of smoothing operators:
\begin{proposition}
  \label{prp:existence-smoothing-op}
  Assume that there are operators
  $\map{\chi_{e,v}}{\HSaux_e^{1/2}}{\HS^1_v}$ such that
  \begin{equation*}
    \Gamma_{v,e} \chi_{e,v} \phi_e = \phi_e, \qquad
    \Gamma_{v,e} \chi_{e',v} \phi_{e'} = 0, \qquad e,e' \in E_v, e\ne e',\;
    v \in V.
  \end{equation*}
  Assume in addition that $C^2=\sup_v \sum_{e \in E_v}
  \normsqr[\HSaux_e^{1/2} \to \HS^1_v]{\chi_{v,e}} < \infty$ then
  \begin{equation*}
    (Bf)_v :=
    \frac 12
    \sum_{e \in E_v} \chi_{e,v} (\Gamma_{v,e} f_v - \Gamma_{v_e,e} f_{v_e})
  \end{equation*}
  defines a smoothing operator.
\end{proposition}
\begin{proof}
  We have to show that $f - Bf \in \HS^1$ whenever $f \in
  \HS^{1,\dec}$, but this follows immediately from the fact that
  $\Gamma_{v,e}(f-Bf) = \frac 12\sum_{w =\bd_\pm e} \Gamma_{w,e} f_w$
  is independent of $v=\bd_\pm e$.  The boundedness of $\map B
  {\HS^{1,\dec}}{\HS^{1,\dec}}$ follows easily.
\end{proof}
One can e.g.~choose $\chi_{e,v}=\SDir_v \pi_{v,e}^*$ under suitable
assumptions on the maps $\pi_{v,e}$ (e.g., $\pi_{v,e}^*(\HSaux_e^
{1/2})\subset \HSaux_v^{1/2}$), where $\SDir_v$ is the Dirichlet
solution operator of $\Pi_v$.  The necessary assumptions are typically
fulfilled in our graph-like manifold example; in particular, if the
boundary components $Y_e$, $e \in E_v$, do not touch in $\bd X_v$, see
\Ex{vx-coupling}.

Finally, we show the norm bound~\eqref{eq:smoothing-op-small} of the
smoothing operators under a slightly stronger assumption than the
$\delta$-closeness of $\Gamma$ and $\wt\Gamma$ (see
\Def{bd-maps.close}):
\begin{proposition}
  Assume that a smoothing operator $\map {\wt B}{\wt \HS^{1,\dec}}{\wt
    \HS^{1,\dec}}$ as in \Prp{existence-smoothing-op} exists and that
  there is $\delta>0$ such that
  \begin{equation*}
    \sum_{e \in E_v} \bignormsqr[\wt \HSaux_e^{1/2}]
          {\wt \pi_{v,e}(\wt \Gamma_v J_v^1 - I_v \Gamma_v)f_v}
     \le \delta^2 \normsqr[\HS^1_v]{f_v}
  \end{equation*}
  holds for all $v \in V$ and $f_v \in \HS^1_v$.  Then $\norm[\wt
  \HS^{1,\dec}] {\wt B J^{1,\dec} f} \le \delta C \norm[\HS^1] f$ holds
  for all $f \in \HS^1$.
\end{proposition}
\begin{proof}
  We have
  \begin{multline*}
    \normsqr[\wt \HS^{1,\dec}] {\wt B J^{1,\dec} f}
    = \sum_{v \in V}\Bignormsqr[\wt \HS^1_v]
      {\frac 12 \sum_{e \in E_v}
          \wt \chi_{e,v}
          \bigl( 
             \wt \Gamma_{v,e} J^1_v f_v - \wt \Gamma_{v_e,e} J^1_{v_e} f_{v_e} 
          \bigr)
      }\\
    = \frac 14 \!\sum_{v \in V}\Bignormsqr[\wt \HS^1_v]
      {\sum_{e \in E_v}
        \wt \chi_{e,v}\Bigl(
           \bigl(\wt \Gamma_{v,e} J^1_v \!-\! 
                      \wt \pi_{v,e} I_v \Gamma_{v,e}\bigr) f_v
         + 
           \bigl( \wt \pi_{v_e,e} I_{v_e} \Gamma_{v_e,e}
               - \wt \Gamma_{v_e,e} J^1_ {v_e}\bigr) f_{v_e} 
        \Bigr)
      }\\
      \le  C^2 \sum_{v \in V} \sum_{e \in E_v}
        \bignormsqr[\wt \HSaux^{1/2}_e] 
            {\wt \pi_{v,e}
              \bigl(\wt \Gamma_v J^1_v - I_v \Gamma_{v,e}\bigr) f_v}
      \le C^2 \delta^2 \normsqr[\HS^1] f
  \end{multline*}
  where we used that $\Gamma_{v,e} f_v=\Gamma_{v_e,e}f_{v_e}$ for the
  second equality.
\end{proof}

A careful observer might know that the following quote is not a rude
reminder of the discomfort of aging, but just a quote from Pavel's web
page \dots
\paragraph{Epilogue.}
\czQuote{Hl\'\i dejte si ty vz\'acn\'e okam\v ziky, kdy v\'am to je\v
  st\v e mysl\'\i. Mohou b\'yt posledn\'\i \dots}

%
%
\newcommand{\etalchar}[1]{$^{#1}$}
\providecommand{\bysame}{\leavevmode\hbox to3em{\hrulefill}\thinspace}

\end{document}